\documentclass[11pt]{amsart}

\usepackage{geometry}                
\usepackage{amssymb}
\usepackage{epstopdf}
\usepackage{ifthen}
\usepackage{xcolor}
\usepackage{cite}
\usepackage{hyperref}
\usepackage{graphicx}
\graphicspath{{./gfx/}}

\newtheorem{question}{Question}
\newtheorem{lemma}{Lemma}
\newtheorem{theorem}[lemma]{Theorem}

\newtheorem{proposition}[lemma]{Proposition}
\theoremstyle{definition}
\newtheorem{remark}[lemma]{Remark}
\newtheorem{definition}[lemma]{Definition}
\newtheorem{example}[lemma]{Example}
\newtheorem*{example*}{Example}

\newcommand{\exclude}[1]{}

\definecolor{grey}{rgb}{0.7, 0.7, 0.7}
\definecolor{darkgreen}{rgb}{0., 0.6, 0.}

\ifthenelse{\isundefined{\draft}}{%
	\newcommand{\obsolete}[1]{}
	\newcommand{\detail}[1]{}
	\newcommand{\new}[1]{#1}
	\newenvironment{newenv}{}{}
	\newcommand{\todo}[1]{}
}{
	\newcommand{\obsolete}[1]{{\color{grey}\renewcommand{\color}[1]{} #1}}
	\newcommand{\detail}[1]{{\color{purple}\scriptsize  #1 }}

	\newcommand{\new}[1]{{\color{darkgreen}#1}}
	\newenvironment{newenv}{\color{darkgreen}}{}
	\newcommand{\todo}[1]{{\color{blue}TODO: #1}}
}

\newcommand{\E}{{\mathbb{E}}}
\newcommand{\N}{{\mathbb{N}}}
\renewcommand{\P}{{\mathbb{P}}}
\newcommand{\R}{{\mathbb{R}}}
\newcommand{\V}{{\mathbb{V}}}
\newcommand{\Beta}{\operatorname{B}}
\newcommand{\median}{\operatorname{median}}

\title{A dice game, a multinomial walk, and the inverted Dirichlet distribution}
\author{Gunther Leobacher and Alexander Steinicke}

\date{}                                           

\begin{document}
\maketitle

\begin{abstract}
We consider a simple dice game, which leads to an intriguing study of multinomial walks, with surprising and seemingly paradoxical
properties.
The winning and losing probabilities of a general version of the game are investigated via conjugacy relations between Gamma and Poisson distributions, 
as well as between  negative multinomial and inverted Dirichlet distributions.
We show a monotonicity property of the regularized beta function, which implies a monotonicity property of the winning probability.
Furthermore, the asymptotic behavior of the game for one or several parameters of the game tending to infinity is analyzed, as well as the probability of being last in the game. 
\end{abstract}

\todo{
\begin{itemize}
\item Find appropriate title
\begin{itemize}
\item An advertisement for the inverted Dirichlet distribution
\item Multinomial walks and the inverted Dirichlet distribution
\item Are inverted Dirichlet distributions a child's game?
\item A dice game, a paradoxon, and the inverted Dirichlet distribution
\item A dice game, a multinomial walk, and the inverted Dirichlet distribution
\item A paradoxic dice game given by a random walk
\end{itemize}
\end{itemize}

}


\ifthenelse{\isundefined{\draft}}{%
}
{
\tableofcontents
}

\section{Introduction}\label{sec:introduction}

Consider the following game: There are 11 players numbered from \(1\) to \(11\), and each player \(\ell\) is assigned a goal \(n_\ell:=\min(\ell,12-\ell)\). In each round of the game, two dice are thrown. Denote the sum of the two dice 
in round \(k\) by \(X_k\). Then in round \(k\), player number \(X_k-1\) moves 1 step forward. The first player to reach their goal of \(n_\ell\)
steps is the winner. See Figure \ref{fig:board} for the game board and Table \ref{tbl:adv-probs} for the 
probability for every player to advance in a given round of the game, which equals \(\frac{1}{n_\ell}\) for player \(\ell\).  
 
\begin{figure}[h]
\includegraphics[width=9cm]{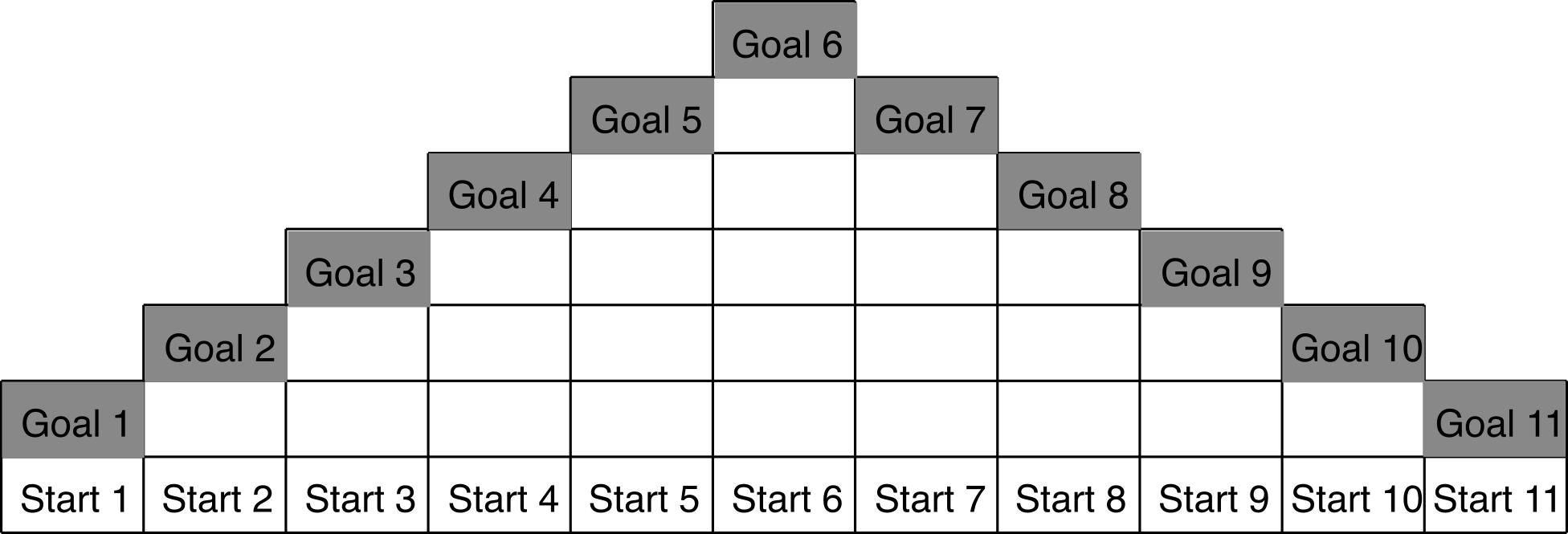}
\caption{The game board}\label{fig:board}
\end{figure}

\begin{table}[h]
\begin{tabular}{l{|}ccccccccccc}
Player number&1&2&3&4&5&6&7&8&9&10&11\\ \hline
Steps needed to win & 1&2&3&4&5&6&5&4&3&2&1\\\hline
Probability of a step in each round & \(\frac{1}{36}\)& \(\frac{2}{36}\)& \(\frac{3}{36}\)& \(\frac{4}{36}\)& \(\frac{5}{36}\)& \(\frac{6}{36}\)& \(\frac{5}{36}\)& \(\frac{4}{36}\)& \(\frac{3}{36}\)& \(\frac{2}{36}\)& \(\frac{1}{36}\)
\end{tabular}
\medskip 
\caption{The probabilities for each player to advance}\label{tbl:adv-probs}
\end{table}

The question is: Which 
of the players is/are most likely to win?

For this, let \(N^\ell\) denote the number of rounds played until player \(\ell\)  reaches their goal.
Then \(N^\ell-n_\ell\) follows a negative binomial distribution with parameters 
\(n_\ell\) and  \(\frac{n_\ell}{36}\), since it is the number
of failures before the \(n_\ell\)-th success with the probability of success equal to \(\frac{n_\ell}{36}\):
\[
\P(N^\ell=n_\ell+j)={n_\ell+j-1 \choose n_\ell-1} \Big(\frac{n_\ell}{36}\Big)^{n_\ell}\Big(1-\frac{n_\ell}{36}\Big)^{j}\,,\quad(j\ge 0).
\]
Thus the expected number of rounds needed for player \(\ell\)  to reach their goal is
\[
\E(N^\ell)=\frac{n_\ell}{\frac{n_\ell}{36}}=36\,.
\]
This means, the expected number of rounds played until they reach their respective goals is the same for every player, and one might be tempted 
to conjecture that also the probability of winning the game is the same for every player. That  this is not the case was
discovered by Evita Lerchenberger and Christoph Oberbucher \cite{LeOb26}. In fact, the probabilities of winning, rounded to two decimals, are given 
in Table \ref{tbl:dice-prob}, and they show that the combined probability of one of the players 1 or 11 winning is almost \(\frac 1 2\).

\begin{table}[h]
\begin{tabular}{l{|}ccccccccccc}
Player number&1&2&3&4&5&6\\ \hline
Steps needed to win & 1&2&3&4&5&6\\ \hline
Probability of step & 0.028 & 0.055 & 0.083 & 0.11 & 0.139 & 0.167                                      \\ \hline
Probability to win &0.242&0.111&0.065&0.042&0.030&0.022
\end{tabular}
\smallskip
\caption{The probabilities for each of players 1,\ldots,6 to reach their goal first. The probability for player \(\ell\) with \(\ell>6\) is the same as that for player \(12-\ell\).}\label{tbl:dice-prob}
\end{table}

That is, the players with smaller goals are significantly more likely to win -- raising the question whether this holds for every
game of a similar structure.  Computing the winning probabilities in Table \ref{tbl:dice-prob} is not hard: One can model the game
as a random walk on \(\N^{11}\), where the transition probability from state \(\xi=(\xi_1,\ldots, \xi_{11})\) to \(\eta=(\eta_1,\ldots, \eta_{11})\)
is given by 
\[
p_{\xi,\eta}=
\begin{cases}
 \frac{n_\ell}{36} & \text{if  }\eta=\xi+e_\ell\\
 0 & \text{else}
\end{cases}
\]
(here \(e_\ell\in \N^{11}\) denotes the \(\ell\)-th canonical basis vector. Also, we use the convention \(0\in \N\) throughout the paper)
and the probability of player \(\ell\) winning is equal to the probability that the \(\ell\)-th coordinate becomes \(n_\ell\) first, which can be 
computed by recursion. In fact, we will write down the winning probability of a player concisely with the help of the negative multinomial distribution in Section \ref{sec:gprob}. However, the numerical evaluation of this formula is barely more efficient than the computation of the recursion.  
Lemmas \ref{lem:gammarep} and \ref{lem:conjugate}, also presented in Section \ref{sec:gprob}, lead to a much more efficient method that employs numerical integration
of gamma distribution densities.

\medskip

For the formulation of our main question and theorem, we generalize the earlier game in the following way: 
We have $m$ players, where each needs $n_\ell\in \N_{\ge 1},\, 1\leq \ell\leq m$, steps to reach their respective goal and advances towards 
that goal by one step with probability $p_\ell=\frac{n_\ell}{n_1+\ldots+n_m}$, $1\leq \ell\leq m$ (so $p_1+\dotsb+p_m=1$) in every round. 

\begin{question}\label{question1}
Is it true that the players with the smallest goals (and the smallest probabilities for advancing) have the highest probability to reach their goal before
all others?  
\end{question}

It is not easy to get an intuition for this situation. For example, the question cannot be answered simply by studying the expectations and variances involved. In fact, it is not hard to find nonnegative random random variables \(X,Y\) with either of the following properties: 
\begin{itemize}
\item \(\E(X)\le\E(Y)\) and \(\V(X)\le\V(Y)\)    or
\item  \(\E(X)>\E(Y)\) and \(\V(X)>\V(Y)\)  or
\item  \(\E(X)\le\E(Y)\) and \(\V(X)>\V(Y)\) or
\item  \(\E(X)>\E(Y)\) and \(\V(X)\le\V(Y)\) 
\end{itemize}
and \(\P(X < Y)>\frac 1 2\).

\detail{

\begin{example*}
We set $\P(S=1 \land T=2)=\frac{2}{3}=1-\P(S=7 \land T=5)$. Then, $\E(S)=\frac{2}{3}\cdot 1+ \frac{1}{3}\cdot 7=3$ and $\E(T)=\frac{2}{3}\cdot 2+ \frac{1}{3}\cdot 5=3$, but precisely $\P(S\le T)=\frac{2}{3}$. 
\end{example*}

Note that $S$ has a larger variance than $T$. The next counterexample shows, however, that one cannot conclude $\P(S\le T)\ge \frac{1}{2}$ even from $\E(S)<\E(T)$ and $\V(S)>\V(T)$.

\begin{example*}
We set $\P(S=2 \land T=1.9)=\frac{1}{3}$, $\P(S=4 \land T=4.1)=\frac{1}{3}$, and $\P(S=6 \land T=6)=\frac{1}{3}$. Then, $\E(S)=\E(T)=4$, $\V(S)\approx 2.67<2.81\approx\V(T)$, but $\P(S\le T)=\frac{2}{3}$.
\end{example*}

Does anything change if we require stochastic independence of the times? The next example shows that this is not the case.

\begin{example*}
Let $S=\sqrt{3}-1+S'$, where $S'\sim \exp(1)$ and $T\sim U([0,\sqrt{12}])$. Then, one can easily calculate that $\E(S)=\E(T)=\sqrt{3}$ and $\V(S)=\V(T)=1$. Obviously, $S$ and $T$ are both non-negative. Now
\begin{align*}
\P(S\le T)&=\int_0^\infty \int_0^\infty 1_{[0,t]}(s)f_S(s)f_T(t)ds\,dt\\
&=\int_0^\infty \int_0^\infty 1_{[0,t]}(s)1_{[\sqrt 3-1,\infty)}(s)e^{\sqrt 3-1-s}1_{[0,\sqrt {12}]}(t)\frac{1}{\sqrt {12}}ds\,dt\\
&=\frac{1}{\sqrt {12}}\int_{\sqrt 3-1}^{\sqrt {12}} \int_{\sqrt 3-1}^t  e^{\sqrt 3-1-s} ds\,dt
=\frac{1}{\sqrt {12}}\int_{\sqrt 3-1}^{\sqrt {12}}  1-e^{\sqrt 3-1-t} dt\\
&=\frac{1}{\sqrt {12}}\int_{\sqrt 3-1}^{\sqrt {12}} 1 dt-\frac{1}{\sqrt {12}}\int_{\sqrt 3-1}^{\sqrt {12}}  e^{\sqrt 3-1-t} dt\\
&=\frac{12-(\sqrt 3-1)^2}{2\sqrt {12}}-1+e^{\sqrt 3-1-\sqrt {12}}\approx 0.72\,,
\end{align*} 
which is comfortably larger than $\frac{1}{2}$. We can thus add a small constant $a$ to $S$ to make $\E(a+S)$ a little smaller or larger than $\E(T)$, or multiply $T$ by $1+b$ for small $b$ to make $\V(bT)$ larger or smaller than $\V(S)$, and still have $\P(S\le T)>\frac{1}{2}$.
\end{example*}

In summary, the following conditions are all independent:
\begin{itemize}
\item \(\P(S\le T)>\frac 1 2\)
\item \(\E(S)<\E(T)\)
\item \(\V(S)<\V(T)\)
\end{itemize}
}
Evita Lerchenberger and Christoph Oberbucher \cite{LeOb26} also observed the following `paradox' 
--- conjectured by Mark Lawson (personal communication with E.~Lerchenberger) --- regarding the probabilities for each player to reach their goal  last
(see Table \ref{tbl:dice-prob-last}).

\medskip 
\begin{table}[h]
\begin{tabular}{l{|}ccccccccccc}
Player number&1&2&3&4&5&6\\ \hline
Steps needed to win & 1&2&3&4&5&6\\ \hline
Probability of step & 0.028 & 0.055 & 0.083 & 0.11 & 0.139 & 0.167                                      \\ \hline
Probability to be last & 0.147 & 0.110 & 0.087 & 0.071 & 0.060  & 0.051
\end{tabular}
\smallskip
\caption{The probabilities for each players 1,\ldots,6 to reach their goal last. The probability for player \(\ell\) with \(\ell>6\) is the same as that for player \(12-\ell\).}\label{tbl:dice-prob-last}
\end{table}

One sees that, not only is the probability that players 1 and 11 reach their goals first higher than for all other players, 
also the probability of reaching it last is the highest for the same players! 
The combined probability for  one of these two players being  last, is almost 30\%.

\begin{question}\label{question2}
For \(m\ge 3\), is it true that the players with the smallest goals (and the smallest probabilities for advancing) have the highest probability of reaching their goal after
all others?  
\end{question}

The main motivation of this article is to answer Question \ref{question1} in the affirmative. In Section \ref{sec:twoplayers} we consider
the 2-player case, where we give different approaches for answering Question \ref{question1}, and we discuss which approach is 
promising for \(m\ge 3\). It turns out that one should look for a multidimensional version of the inverted beta distribution, and that the 
proper generalization for our purpose is the inverted Dirichlet distribution, which we recall in Section \ref{sec:gprob}. Furthermore, 
it is shown how to represent the winning probabilities in terms of this distribution, as well as by means of independent gamma- or Poisson 
random variables. One separate result of  Section \ref{sec:gprob} is Theorem \ref{thm:pi-bij}, which states that there is a (even diffeomorphic)
bijective correspondence between the advancing probabilities of the players and their respective probabilities of winning.

In Section \ref{sec:winning} we use the findings from Section \ref{sec:gprob} to answer Question \ref{question1}. For this, a further 
result is needed, Lemma \ref{th:inc-reg-beta-dec}, which is a monotonicity result for the incomplete regularized beta function.

Section \ref{sec:asymptotics} is devoted to the investigation of the limiting behavior of the game if one or several 
of the goals \(n_1,\ldots,n_m\) tend to infinity. There is also a surprise waiting there, namely that our game yields a `real world' 
example of a  sequence with index set \(\N_{\ge 1}^m\) where a different order in which limits are taken gives different results.

\medskip

It is obvious, that the assertion of Question \ref{question2} is false for \(m=2\), since in that case the probability of being last 
is just the probability of not being first.
In Section \ref{sec:losing-prob} we discuss Question \ref{question2} for \(m\ge 3\) in some detail and give more numerical examples. 
It will turn out that, in the form stated above, the answer to Question \ref{question2} is `no'.
However, we show that asymptotically, that is, for \(n_1,\ldots,n_m\) sufficiently large, the answer to Question \ref{question2} still is `yes'.

\medskip

We want to highlight the contributions within this article to several fields on mathematics.
Firstly, it is an intriguing contribution to the theory of random 
walks on the integer lattice, as famously pioneered by George Polya in \cite{Polya1921aa}. In contrast to the mentioned work, the random 
walk we consider here does not visit any state more often than once, so questions of recurrence are vacuous (However,
we do give a characterization for recurrence of a related walk in Remark \ref{rem:tildeR}, with recurrence for \(m\le 3\), corresponding to dimensions 1 and 2, and transience for \(m\ge 4\).).  On the
other hand, questions about hitting probabilities of (parts of) hyperplanes are also interesting in this setup.

Second, our computation of the winning/losing probabilities via independent gamma random variables gives rise to a 
numerical method for computing the hitting probabilities above, which 
is so immensely faster than the straightforward recursive calculation, that it merits being reported to a wider public.

Third, our calculations brought about a useful result on incomplete beta functions that seems to have escaped the mathematical
community so far.

\section{Two players}\label{sec:twoplayers}

Before tackling the general problem, we solve the special case $m=2$, so $p_1=\frac{n_1}{n_1+n_2}$ and $p_2=\frac{n_2}{n_1+n_2}$. 

In fact, we will give 
three proofs that player 1 wins with probability greater than \(\frac 1 2\) iff \(n_1<n_2\). Note that it follows by symmetry 
that the probability of player 1 winning is \(\frac 1 2\) if \(n_1=n_2\).

For the first proof, consider the random walk on \(\N^2\) which moves one step to the right with probability \(p_1\) and one step up 
with probability \(p_2\) in every round. The probability for player 1 winning equals the one of this random walk reaching the set \(\{(x,y): x\ge n_1\}\) before the set 
 \(\{(x,y): y\ge n_2\}\). This in turn is the same as that of the walk crossing the line \(\{(x,y): x+y=n_1+n_2-1\}\) in a point \((j,k)\) with
 \(0\le k<n_2\) (see Figure \ref{fig:m_eq_2}). The walk needs precisely \(n_1+n_2-1\) steps to cross that line, and hitting it in the point \((j,k)\) means that 
 the walk has made precisely \(j\) steps to the right and \(k\) steps up.  Therefore, the probability of player 1 winning is
 \[
 \pi_1(n_1,n_2):=\sum_{k=0}^{n_2-1}{n_1+n_2-1 \choose k}p_2^k\ p_1^{n_1+n_2-1-k},
 \]
 the probability that a random variable \(X\), having binomial distribution with parameters \(n_1+n_2-1\) and \(p_2\),  is smaller than \(n_2\).
 It is known that the median of \(X\) is unique if \(p_2\) is rational, \cite[Theorem 4.1]{Nowakowski2021}, and satisfies 
  \(\lfloor(n_1+n_2-1)p_2)\rfloor\le \operatorname{median}(X)\le \lceil(n_1+n_2-1)p_2)\rceil\) \cite[Corollary 1]{KaasBuhr}. In our case, where \(p_2=\frac{n_2}{n_1+n_2}\) this means that \(\operatorname{median}(X)\) is unique and \(n_2-1\le \operatorname{median}(X)\le n_2\). 
 Therefore \(\pi_1(n_1,n_2)=\P(X\le n_2 -1)=1-\P(X> n_2 )\ge 1-\P(X>\operatorname{median}(X))\ge \frac 1 2\). Now, if we had 
 \( \pi_1(n_1,n_2)=\frac 1 2\), 
 then \(\frac 1 2 = \P(X\le n_2 -1) =\P(X\ge n_2)\) so that every number \(x\in (n_2-1,n_2)\) is a median of \(X\). But this contradicts the uniqueness of the median.\qed

\begin{figure}
\begingroup%
  \makeatletter%
  \providecommand\color[2][]{%
    \errmessage{(Inkscape) Color is used for the text in Inkscape, but the package 'color.sty' is not loaded}%
    \renewcommand\color[2][]{}%
  }%
  \providecommand\transparent[1]{%
    \errmessage{(Inkscape) Transparency is used (non-zero) for the text in Inkscape, but the package 'transparent.sty' is not loaded}%
    \renewcommand\transparent[1]{}%
  }%
  \providecommand\rotatebox[2]{#2}%
  \newcommand*\fsize{\dimexpr\f@size pt\relax}%
  \newcommand*\lineheight[1]{\fontsize{\fsize}{#1\fsize}\selectfont}%
  \ifx\svgwidth\undefined%
    \setlength{\unitlength}{283.90144649bp}%
    \ifx\svgscale\undefined%
      \relax%
    \else%
      \setlength{\unitlength}{\unitlength * \real{\svgscale}}%
    \fi%
  \else%
    \setlength{\unitlength}{\svgwidth}%
  \fi%
  \global\let\svgwidth\undefined%
  \global\let\svgscale\undefined%
  \makeatother%
  \begin{picture}(1,0.72447974)%
    \lineheight{1}%
    \setlength\tabcolsep{0pt}%
    \put(0,0){\includegraphics[width=\unitlength,page=1]{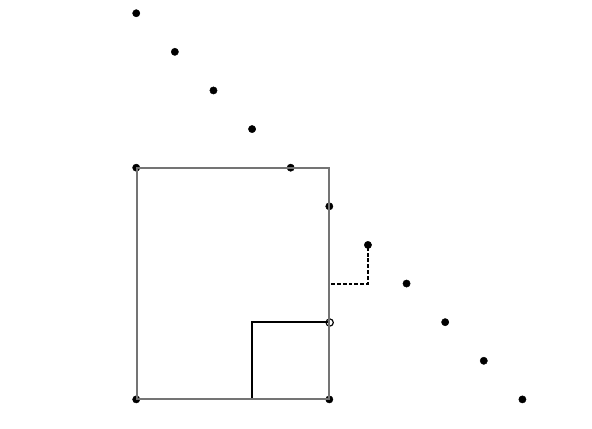}}%
    \put(0.56227635,0.00499236){\makebox(0,0)[t]{\lineheight{1.25}\smash{\begin{tabular}[t]{c}$n_1$\end{tabular}}}}%
    \put(0.17928043,0.43321848){\makebox(0,0)[t]{\lineheight{1.25}\smash{\begin{tabular}[t]{c}$n_2$\end{tabular}}}}%
    \put(0.20638778,0.02213909){\makebox(0,0)[t]{\lineheight{1.25}\smash{\begin{tabular}[t]{c}$0$\end{tabular}}}}%
    \put(0.88325615,0.00499236){\makebox(0,0)[t]{\lineheight{1.25}\smash{\begin{tabular}[t]{c}$n_1+n_2-1$\end{tabular}}}}%
    \put(0.11701692,0.69351181){\makebox(0,0)[t]{\lineheight{1.25}\smash{\begin{tabular}[t]{c}$n_1+n_2-1$\end{tabular}}}}%
  \end{picture}%
\endgroup%
\caption{Any walk reaching \((n_1,y)\) before \((x,n_2)\) crosses the line \(x+y=n_1+n_2-1\) in a point \(x',y'\) where \(y'<n_2\).} \label{fig:m_eq_2}
\end{figure}
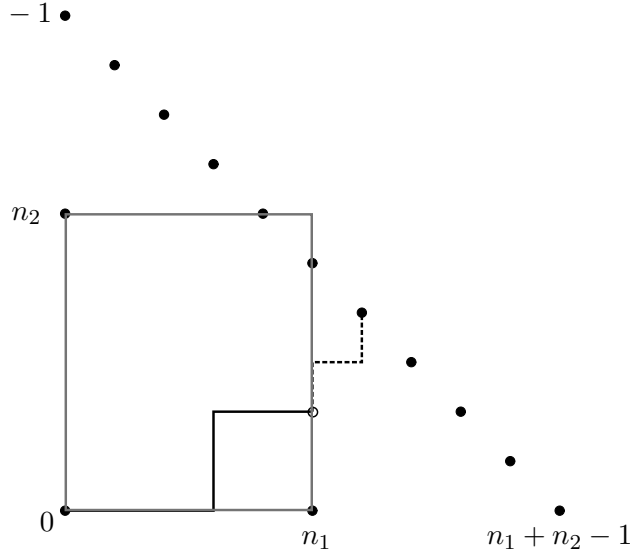

\medskip
The second proof starts with noting that the probability for the first player winning can be expressed by the CDF of a negative binomial distribution,
\begin{align}\label{eq:neg-bin-m-eq-2}
\pi_1(n_1,n_2)&=\mathbb{P}(A_2\leq n_2-1)=\Big(\frac{n_1}{n_1+n_2}\Big)^{n_1}\sum_{k_2=0}^{n_2-1}\binom{n_1+k_2-1}{n_1-1,k_2} \Big(\frac{n_2}{n_1+n_2}\Big)^{k_2},
\end{align}
where \(A_2\) counts the number of `failures' (advances of player 2) before the \(n_1\)-th `success' (advances of player 1).
Again, we want to show that if \(n_2>n_1\) then \(\pi_1(n_1,n_2)>\frac{1}{2}\).

A viable way to tackle the problem is to apply known bounds for the negative binomial distribution, using results from \cite{goeb}. There, in \cite[Theorem (3.8)]{goeb}, the author states that the median of a negative binomial distribution  satisfies for $0<p_1<p_2=1-p_1<1$,
\begin{align*}
&\median(A_2)=\Big\lceil -\frac{\log(2)}{\log(p_2)}-1 \Big\rceil,\quad\text{for }n_1=1,\\
&\Big\lfloor (n_1-1)\frac{p_2}{p_1} \Big\rfloor+1\leq \median(A_2)\leq \Big\lceil n_1\frac{p_2}{p_1} \Big\rceil-1,\quad\text{for }n_1>1.
\end{align*}
Inserting our probabilities, the above states $\median(A_2)=\Big\lceil -\frac{\log(2)}{\log(n_2)-\log(1+n_2)}-1 \Big\rceil$ in the case $n_1=1$. It is readily checked that for $n_2>1$, $\median(A_2)<n_2-1$ follows from the estimate 
\(\log(1+x)\ge x\log(2)\) for all \(x\in [0,1]\), so $\pi_1(n_1,n_2)=\mathbb{P}(A_2\leq n_2-1)>\frac{1}{2}$. In the case $n_1>1$, we obtain the upper bound $n_2-1$ for the median, so $\pi_1(n_1,n_2) =\mathbb{P}(A_2\leq n_2-1)\geq \pi_2(n_1,n_2)$.
To show the median cannot attain the value $n_2-1$, we use the main result in \cite{Nowakowski2021}, which states that for binomial distributions with rational $p_1\neq p_2$, a strong median exists, i.e.~the cumulative distribution function (CDF) does not assume the value $\frac{1}{2}$. Since the CDF of a negative binomial distribution can be obtained by the one of a binomial distribution and suitable parameters (in particular $p_1$ and $p_2$ are switched, see e.g.~\cite[(3.5)]{goeb} or \cite{Morris63}), we get that also the CDF of a negative binomial distribution does not attain $\frac{1}{2}$, so $\pi_1(n_1,n_2)>\frac{1}{2}>\pi_2(n_1,n_2)$.\qed
\medskip

A third proof directly applies a recent result from the literature about the incomplete regularized beta function. 
As before, we express the winning probability of player 1 through the CDF of a negative binomial distribution, see \eqref{eq:neg-bin-m-eq-2}. It is well known (see, e.g., \cite[Chapter 5, Section 6]{JohnKotz92} for a classical reference, but this can easily be proven by repeated partial integration) that the cumulative distribution function of a negative binomial distribution is given by the incomplete regularized
beta function. In our case we get
\begin{equation}\label{eq:beta_prob}
\pi_1(n_1,n_2)=I_{\frac{n_1}{n_1+n_2}}(n_1,n_2)=\frac{1}{\Beta(n_1,n_2)}\int_0^{\frac{n_1}{n_1+n_2}}t^{n_1-1}(1-t)^{n_2-1}dt,
\end{equation}
where \(\Beta\) denotes Euler's beta function.
One of the main results in \cite{Hornik24} is that $I_{\frac{n_1}{n_1+n_2}}(n_1,n_2)$ is non-increasing in $n_1$. Therefore, as for $n_1=n_2$ the probabilities $\pi_1$ and $\pi_2$ are equal, $\pi_1(n_1,n_2)$ decreases with increasing $n_1$ (and increases when \(n_1\) decreases).\qed

\begin{remark}
The presented approaches for the proof are clearly not independent. In particular, Figure \ref{fig:m_eq_2} illustrates
that \(\P(X \le n_2-1)=\P(Y\le n_2-1)\) if \(X\) has a binomial distribution with \(n_1+n_2-1\) trials and success probability \(p_2\) and  \(Y\) has a negative binomial distribution, counting the number of failures before the \(n_1\)-th success, where the success probability is 
\(p_1=1-p_2\). 

It turns out that writing the probabilities in the way of equation \eqref{eq:beta_prob} in the third approach, leads to a generalization for general \(m\). The reason why the first proof is not directly generalizable can also be
read off of Figure \ref{fig:cuboid-m3}: If one takes \(m=3\), then one cannot conclude through which side of the cuboid the walk went
from the point were it crossed the plane \(x_1+x_2+x_3=n_1+n_2+n_3-1\). 
\begin{figure}[h]
\includegraphics[width=7cm]{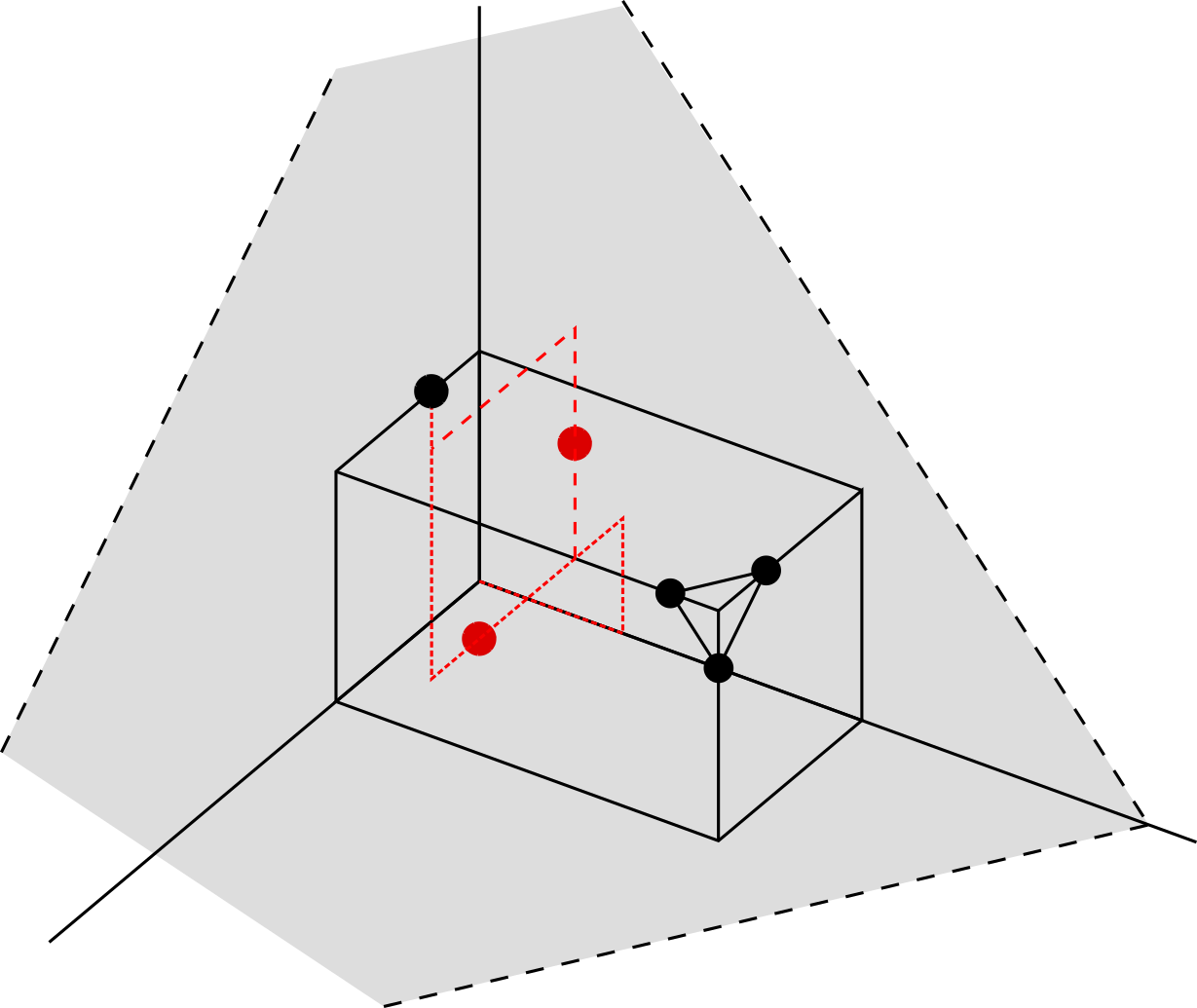}
\caption{Note that the winning probability for player 1 is the hitting probability of one face of an \(m\)-dimensional cuboid. For \(m=3\) there can be paths exiting from different faces of the cuboid but still hitting the plane \(x+y+z=n_1+n_2+n_3-1\) in the same point.}\label{fig:cuboid-m3}
\end{figure}
This also means that the CDF of the negative multinomial distribution can only be expressed by that of a multinomial distribution,
like in \cite{goeb,Morris63}, in dimension one.
\end{remark}

\begin{remark} The fact that $\pi_1(n_1,n_2)>\pi_2(n_1,n_2)$ if $n_1<n_2$ is also a consequence of Lemma \ref{th:inc-reg-beta-dec} from Section \ref{sec:gprob} (setting \(m=2\) and $K=N=0$).  
This lemma even gives a stronger result, since it actually shows that $\pi_1(n_1,n_2)$ is {\em strictly} increasing in $n_2$ (and then the claim follows by symmetry, $\pi_1(n_1,n_2)=\pi_2(n_1,n_2)$ for $n_1=n_2$). 
For the case that  $n_2$ is an integer, this special case of Lemma \ref{th:inc-reg-beta-dec}, stating that $n_2\mapsto I_{\frac{n_1}{n_1+n_2}}(n_1,n_2)$ is increasing, also contains the case of \cite[Theorem 1]{Hornik24}.
\end{remark}

\begin{remark} Our result for $m=2$ complements equation (9) of \cite{Raab1997}, which was treated already in \cite{Raab_1984} and originally found by Vietoris \cite{Vietoris82}, stating that for $n_1, n_2>0$,
\begin{align*}
I_{\frac{n_1}{n_1+n_2}}(n_1,n_2+1)>\frac{1}{2}.
\end{align*}
For comparison, our results imply
\begin{align*}
I_{\frac{n_1}{n_1+n_2}}(n_1,n_2)>\frac{1}{2},\quad\text{if }n_1<n_2,\quad\text{and}\quad I_{\frac{n_1}{n_1+n_2}}(n_1,n_2)<\frac{1}{2},\quad\text{if }n_1>n_2,
\end{align*}
and, of course, equality holds for $n_1=n_2$. 
\end{remark}

\section{Multinomial walks and the inverted Dirichlet distribution}\label{sec:gprob}

Before we answer Question \ref{question1} put forward in Section \ref{sec:introduction}, we present a useful alternative representation of the winning probabilities
in the game with \(m\) players and general probabilities for advancing. This representation can be
seen as an \(m\)-dimensional version of  \eqref{eq:beta_prob}. Let \(m\ge 2\) be an integer and \(n_1,\ldots,n_m\) be positive integers. 
We are interested in a game with $m$ players, where in each round one of the players is chosen at random to advance one step towards their goal. All players start at 0, the goal of player \(\ell\), $1\leq \ell\leq m$, is $n_\ell$ and the probability for player \(\ell\)
to be chosen in a given round is \(p_\ell\),  with $p_1+\dotsb+p_m=1$. 
The winner of the game is the player who reaches their goal first (with the winner being uniquely determined if all \(n_\ell\) are positive). 
By symmetry, it suffices to study the winning probability for $\ell=1$. 

Let \(R_n=(R_{1,n},\ldots, R_{m,n})\) denote the state of the game in round \(n\ge 0\),  that is, \(R_{\ell,n}\) is the number of steps 
player \(\ell\) has advanced toward their goal up until round \(n\). 
Then the process \(R=(R_1,\ldots, R_m)\) is a random walk with state space  \(\N^m\) starting in \(0\in \N^m\) where the transition
probability from state \(\xi=(\xi_1,\ldots, \xi_{m})\) to \(\eta=(\eta_1,\ldots, \eta_{m})\)
is given by (where \(e_\ell\in \N^{m}\) denotes the \(\ell\)-th canonical basis vector)
\[
p_{\xi,\eta}=
\begin{cases}
p_\ell & \text{if  }\eta=\xi+e_\ell,\\
 0 & \text{else}.
\end{cases}
\]

\begin{remark} \label{rem:tildeR}
Consider the process \(\tilde R=(\tilde R_n)_{n\in \N}:=(R_n-n m^{-1}1_{\R^m})_{n\in \N}\), where \(1_{\R^m}\in \R^m\) denotes the vector with all entries equal to 1. 
Then \(\tilde R\) is a random walk on the lattice generated by the vectors
\((e_1-m^{-1}1_{\R^m},\ldots,e_m-m^{-1}1_{\R^m})\), which span an \((m-1)\)-dimensional hyperplane. Write \(\tilde e_\ell:=e_\ell -m^{-1}1_{\R^m}\). 
We now study the recurrence/transience of this 
random walk.

The walk $\tilde R$ can return to 0 only after a number \(n\) of steps which is a multiple of \(m\), i.e., \(n=mk\). 
Of these steps, there must be precisely $k$ in each of the directions  \(\tilde e_1,\ldots,\tilde e_m\).
So we get for the probability of returning to 0 in \(mk\) steps
\begin{align*}
p_{00}^{mk}&={mk\choose \underset{\text{\(m\) times}}{\underbrace{k,\ldots,k}}}p_1^k\dotsb p_m^k.
\end{align*}
Stirling's formula gives $\frac{(mk)!}{k!\ldots!k!}\asymp \frac{\sqrt{2m k \pi}(mk)^{mk}e^{-mk}}{\left(\sqrt{2k\pi}k^ke^{-k}\right)^m}
=\frac{\sqrt{m}}{\left(\sqrt{2\pi}\right)^{m-1}}\frac{1}{k^{\frac{m-1}{2}}}m^{mk}$ for \(k\to \infty\).

Write $\eta:=\Big(\prod_{j=1}^m mp_j\Big)^{\frac{1}{m}}$. Then by the arithmetic-geometric mean inequality
$$\eta=\Big(\prod_{j=1}^m mp_j\Big)^{\frac{1}{m}}\le \frac{1}{m}\sum_{j=1}^m mp_j=1$$
with equality iff $p_j= \frac{1}{m}$.
Therefore,  if $p_j\ne  \frac{1}{m}$ for at least one $j\in \{1,\dotsc,m\}$ 
\[
p_{00}^{mk}\asymp \frac{\sqrt{m}}{\left(\sqrt{2\pi}\right)^{m-1}}\frac{1}{k^{\frac{m-1}{2}}}\eta^{mk}<\frac{\sqrt{m}}{\left(\sqrt{2\pi}\right)^{m-1}}\eta^{mk},
\] is a summable sequence in $k$, so by \cite[Theorem 1.5.3]{norris97} the point $0$ (and therefore the random walk) is transient. 

Otherwise if all $p_j=\frac{1}{m}, j\in \{1,\dotsc,m\}$ we get that \(
p_{00}^{mk}\asymp \frac{\sqrt{m}}{\left(\sqrt{2\pi}\right)^{m-1}}\frac{1}{k^{\frac{m-1}{2}}}.
\)
So $\sum_{n\in \N_{\ge 0}}p_{00}^{(n)}=\sum_{k\in \N_{\ge 0}}p_{00}^{(mk)}$ is infinite for \(m\le 3\) and finite for \(m\ge 4\). 
Thus by \cite[Theorem 1.5.3]{norris97} the point $0$ is recurrent for \(m\le 3\) and transient for \(m\ge 4\).

Note that this is in correspondence with the classical results by G.~Polya \cite{Polya1921aa}.
\end{remark}

\begin{remark} 
The process \(R\) also appears in applications, such as mathematical finance.
Concretely, in the {\em multinomial market model} \cite{he1990}, a generalization of the Cox-Ross-Rubinstein binomial model, there are \(m-1\) risky assets 
(plus 1 riskless asset, which we ignore here) with discounted prices \(S^1_n,\ldots, S^{m-1}_n>0\) at times \(n\in \N\), that evolve in the following way: \(S^j_{n+1}=S^j_n Y^j_{n+1}\) where 
\(Y_1,Y_2,\ldots\) are i.i.d.~random vectors with values in \(\{v_1,\ldots,v_{m}\}\subseteq \R_{>0}^{m-1}\). Therefore 
\(\P(S^j_{n+1}=S^j_n v^j_\ell)=p_{\ell}:=\P(Y^j_{n+1}=v_\ell)>0\), \(n\in \N, j\in \{1,\ldots,m-1\}\), \(\ell\in \{1,\ldots,{m}\}\). Taking the component-wise logarithm
gives 
\[
\P\big(\log(S^j_{n+1})=\log(S^j_n)+ \log(v^j_\ell)\big)=p_\ell
\] 
respectively
\[
\P\big(\log(S_{n+1})=\log(S_n)+ L e_\ell\big)=p_\ell
\] 
where 
\(e_\ell\) denotes the \(\ell\)-th canonical basis vector, and \(L\) is the matrix
\[
L=\begin{pmatrix}
\log(v^1_1)&\cdots&\log(v^1_m)\\
\vdots&\ddots&\vdots\\
\log(v^{m-1}_1)&\cdots&\log(v^{m-1}_m)\\
\end{pmatrix}\,.
\]
With this notation, \(\log(S_{n})=\log(S_0)+L R_n\).
In \cite{he1990} this model is studied in detail (see also \cite{Boyle89, ChenChungYang02, Sierag13}). In particular, the convergence of asset prices in the multinomial model to those in the corresponding Black-Scholes model is proven. For this, it is shown in \cite{he1990} that a properly rescaled and projected version of \(R\) converges to an \(m-1\)-dimensional standard Brownian motion. This is in line with the appearance of normal variates in 
Section \ref{sec:asymptotics}.
\end{remark}
 
\new{
\begin{definition}\label{defi:pitau} Consider paths of the random walk \(R\) defined above. 
\begin{enumerate} 
\item The winning probability of player \(\ell\), 
i.e., the probability that \(R_{n,\ell}\) reaches \(n_\ell\) before any \(R_{n,k}\) reaches \(n_k\), \(k\ne \ell\),
is 
\[
\pi_\ell:=\P\Big(\bigcup_{n=1}^\infty \big\{R_{n,\ell}=n_\ell,R_{n-1,\ell}=n_\ell-1,R_{n,k}<n_k,  1\leq k\leq m,k\ne \ell\big\}\Big).
\]
\item The probability of player \(\ell\) being last (for short: losing probability of player \(\ell\)), 
i.e., the probability that \(R_{n,\ell}\) reaches \(n_\ell\) after each \(R_{n,k}\) reaches \(n_k\), \(k\ne \ell\),
is 
\[
\tau_\ell:=\P\Big(\bigcup_{n=1}^\infty \big\{R_{n,\ell}=n_\ell,R_{n-1,\ell}=n_\ell-1,R_{n-1,k}\ge n_k,  1\leq k\leq m,k\ne \ell\big\}\Big).
\]
\end{enumerate}
\end{definition}

\begin{remark}\label{rem:cones} Both the winning probability in the game as well as the probability of being last have a geometric interpretation.
Consider the cone
\[
C_W:=\bigcap_{\ell=1}^m\{(x_1,\ldots,x_m)\in \R^m\colon x_\ell\le n_\ell-\tfrac 1 2\}.
\]
The random walk \(R\) starts in \(0\in C_W\) and leaves \(C_W\) 
 through precisely one of its \((m-1)\)-dimensional faces. The probability of player \(\ell\) winning is the probability of leaving through the  
 \(\ell\)-th face \(\{(x_1,\ldots,x_m)\in \R^m\colon x_\ell= n_\ell\}\cap C_W\). 
\new{
In other words, the  probability of player \(\ell\) winning is the probability of hitting the half space 
\(\{(x_1,\ldots,x_m)\in \R^m\colon x_\ell\ge n_\ell-\tfrac 1 2\}\) at the same time as the complement of \(C_W\).
}

On the other hand, consider the cone 
\[
C_L:=\bigcap_{\ell=1}^m\{(x_1,\ldots,x_m)\in \R^m\colon x_\ell\ge n_\ell-\tfrac 1 2\}.
\]
With probability 1, the random walk will enter \(C_L\) through precisely one of its \((m-1)\)-dimensional faces.  Player \(\ell\) is last in the game if the random walk enters 
through the \(\ell\)-th face  \(\{(x_1,\ldots,x_m)\in \R^m\colon x_\ell= n_\ell-\tfrac 1 2\}\cap C_L\).
In other words, the  probability of player \(\ell\) being last is the probability of hitting the half space 
\(\{(x_1,\ldots,x_m)\in \R^m\colon x_\ell\ge n_\ell-\tfrac 1 2\}\) at the same time as  \(C_L\).

With this, the winning and losing probabilities can be expressed as probabilities that certain 
hitting times for the random walk coincide.
\end{remark}
}

Given \(n_1,\ldots,n_m\in \N_{\ge 1}\), and \(p_1,\ldots,p_m\in (0,1)\) with \(p_1+\dotsb+p_m=1\), 
we get winning probabilities \(\pi_1,\ldots,\pi_m\). Depending on the purpose we will consider 
\(\pi_1,\ldots,\pi_m\) as functions of \(n_1,\ldots,n_m\) or of \(p_1,\ldots,p_m\) (and analog for \(\tau_1,\ldots,\tau_m\), see Section \ref{sec:losing-prob}).
 
 \medskip
 
 For the next proposition we concentrate on \(\pi_1\). Respective statements can be deduced for 
  \(2\le \ell \le m\) by symmetry.
 
 \begin{proposition}\label{thm:winning-prob-neg}  The winning probability has the form  
 \begin{align*}
 \pi_1&=\mathbb{P}(A_k<n_k,  2\leq k\leq m).
 \end{align*}
where $(A_2,\ldots,A_m)$ has a negative multinomial distribution (see e.g.~\cite[Chapter 36]{JohnKotzBal97}) with parameter set 
$\big(n_1,p_2,\ldots,p_m\big)$, i.e.,
\begin{align}\label{eq:neg-multi}
\pi_1=\Big(1-\sum_{k=2}^m p_k\Big)^{n_1}\sum_{k_2=0}^{n_2-1}\cdots\sum_{k_m=0}^{n_m-1}\binom{n_1+k_2+\dotsb+k_m-1}{n_1-1, k_2,\dotsc ,k_m}\prod_{\ell=2}^m p_\ell^{k_\ell}\,.
\end{align}
\end{proposition}

\begin{proof}  Write \(p_1:=1-\sum_{k=2}^m p_k\).
A winning path for player 1 consists of \(n_1\) steps in direction of \(e_1\) and \(k_2,\ldots,k_m\)
in the respective other directions, where \(k_\ell<n_\ell\) for every \(\ell\) and the last step is taken in direction \(e_1\). 
The probability for one such path with given \(k_2,\ldots,k_m\) is \(p_1^{n_1} \prod_{\ell=2}^m p_\ell^{k_\ell}\), and there are 
\(\binom{n_1+k_2+\dotsb+k_m-1}{n_1-1, k_2,\dotsc ,k_m}\)-many such paths. Thus
\[
\P\big(R_{n,1}=n_1,R_{n-1,1}=n_1-1,R_{n,j}=k_j,  2\leq j \leq m\big)=\binom{n_1+k_2+\dotsb+k_m-1}{n_1-1, k_2,\dotsc ,k_m}p_1^{n_1} \prod_{j=2}^m p_j^{k_j}
\] 
and summing over all \(k_2,\ldots,k_m\) with \(k_j<n_j\), \(k=2,\ldots,m\) gives \eqref{eq:neg-multi} on the right hand side and 
\(\pi_1\) on the left hand side (cf.~Definition \ref{defi:pitau}).
\end{proof}

A fundamental step in showing our main assertion is to rewrite the probability of the first player winning with the help of
 the inverted Dirichlet distribution. This distribution is described in detail in \cite{Tiao_Cuttman1965} as a special case of the distributions appearing in \cite{OlkRub64}, and was further studied, e.g., in \cite{Bdir11,  lingm76, Phil88, Ghorbel2009, Ng11, Fang_et_al_1990, Otto2025, Yass74, Yass76} and applied, for example, in \cite{Bdir12, Bdir13, Lai18, Ling24, Tirdad2015}. The subsequent definition is taken from \cite[Example 6.3]{Fang_et_al_1990}.

\begin{definition}[Inverted Dirichlet distribution]\label{def:idir}\hfill\\
Let \(m\ge 2\) be an integer and \(n_1,\ldots,n_m\in (0,\infty)\).
A random vector $(X_2,\dotsc,X_m)$ with values in \([0,\infty)^{m-1}\) has an \emph{inverted Dirichlet distribution} with parameter set $(n_2,\dotsc,n_m;n_1,\lambda)$, denoted by $(X_2,\dotsc,X_m)\sim \operatorname{ID}(n_2,\dotsc,n_m;n_1,\lambda)$, if it has the joint density
\begin{align*}
(s_2,\dotsc,s_m)\mapsto\frac{1}{\Beta(n_1,\dotsc,n_m)}\frac{\lambda^{n_1}\prod_{k=2}^m s_k^{n_{k}-1}}{\Big(\lambda+\sum_{k=2}^m s_k\Big)^{\sum_{k=1}^m n_k}}
\end{align*}
on $[0,\infty)^{m-1}$.
Here $\Beta(n_1,\dotsc,n_m):=\frac{\prod_{k=1}^m \Gamma(n_k)}{\Gamma\big(\sum_{k=1}^m n_k\big)}$ is the multivariate Euler beta function.

In the case where \(\lambda=1\) we simply write \((X_2,\dotsc,X_m)\sim\operatorname{ID}(n_2,\dotsc,n_m;n_1)\).
\end{definition}

\begin{remark}\label{rem:idir}\hfill
\begin{enumerate}
\item In this article we will use only the case of $\lambda=1$ for the inverted Dirichlet distribution.
\item In the case $m=2$ we obtain the density $\frac{1}{\Beta(n_1,n_2)}\frac{s^{n_2-1}}{(1+s)^{n_1+n_2}}$ of $X_2$. Therefore, its survival function $x\mapsto\mathbb{P}(X_2> x)$ is given by
\begin{align*}
\mathbb{P}(X_2> x)=\frac{1}{\Beta(n_1,n_2)}\int_x^\infty \frac{s^{n_2-1}}{(1+s)^{n_1+n_2}}ds, 
\end{align*}
which can be expressed (substituting $t=\frac{1}{1+s}$) as
\begin{align*}
&\frac{1}{\Beta(n_1,n_2)}\int_x^\infty \frac{s^{n_2-1}}{(1+s)^{n_1+n_2}}ds\\
&=\frac{1}{\Beta(n_1,n_2)}\int_0^{\frac{1}{1+x}}t^{n_1-1}(1-t)^{n_2-1}dt=I_{\frac{1}{1+x}}(n_1,n_2),
\end{align*}
where \(I\) denotes the \emph{regularized (lower) incomplete beta function}. Therefore the survival function of a vector $(X_2,\dotsc,X_m)$ provides a multivariate generalization of the regularized (lower) incomplete beta function.
\item This generalization here is the same as \cite{GittinsMaher} propose. Another way to generalize  the regularized incomplete beta function $I$ to multiple dimensions is given by the CDF of the \emph{Dirichlet distribution} (which generalizes the beta distribution supported on a multidimensional simplex, see e.g.~\cite[Section 2]{Tiao_Cuttman1965} and the references therein, or \cite[Example 6.1]{Fang_et_al_1990}. Up to the normalizing factor, the latter coincides with the generalization given in \cite{Raab1997}.
\item The inverted Dirichlet distribution (as well as the Dirichlet distribution) belong to the class of \emph{multivariate Liouville distributions}, see e.g.~\cite[Chapter 6]{Fang_et_al_1990}.
\item This way, the inverted Dirichlet distribution is a multivariate generalization of the \emph{inverted beta distribution} (see \cite{Tiao_Cuttman1965}).
\item\label{it:idir-marg} Let $(X_2,\dotsc,X_m)$ have an inverted Dirichlet distribution with parameter set $(n_2,\dotsc,n_m;n_1)$. 
Then, $(X_2,\dotsc,X_{m-1})$ has an inverted Dirichlet distribution with parameter set $(n_2,\dotsc,n_{m-1};n_1)$
(see \cite[p.~795, below eq.~(2.14)]{Tiao_Cuttman1965}). That is, the marginals of an inverted 
Dirichlet distribution is an inverted 
Dirichlet distribution, where the parameters of the marginal are obtained by removing the complementary parameters of the original distribution.
\end{enumerate}
\end{remark}

We now state two important results concerning the inverted Dirichlet distribution. Note that whenever we consider gamma distributions, we use the {\em scale parametrization}.

\begin{lemma}\label{lem:gammarep} Let \(m\ge 2\) be an integer and \(n_1,\ldots,n_m\in (0,\infty)\).
Let $(G_1,\dotsc,G_m)$ be a vector of independent random variables with $G_k\sim\operatorname{Gamma}(n_k,1)$, then
\begin{align*}
\left(\frac{G_2}{G_1},\dotsc,\frac{G_m}{G_1}\right)\sim \operatorname{ID}(n_2,\dotsc,n_m;n_1)
\end{align*}
and 
\begin{align*}
\left(\frac{G_1}{\sum_{k=1}^m G_k},\dotsc,\frac{G_m}{\sum_{k=1}^m G_k}\right)\sim \operatorname{D}(n_1,\dotsc,n_m),
\end{align*}
where \(\operatorname{D}(n_1,\dotsc,n_m)\) denotes the Dirichlet distribution with parameter set \((n_1,\dotsc,n_m)\) (see, e.g., \cite[Example 6.1]{Fang_et_al_1990} for the definition).

In particular, if $(V_1,\dotsc,V_m)$ is Dirichlet distributed with parameter set $(n_1,\dotsc,n_m)$ then, defining 
\begin{align*}
X_k:=\frac{V_k}{1-V_2-\dotsb-V_m},\quad 2\leq k\leq m,
\end{align*}
we get $(X_2,\dotsc,X_m)\sim \operatorname{ID}(n_2,\dotsc,n_m;n_1)$. 

Conversely, if $(X_2,\dotsc,X_m)\sim \operatorname{ID}(n_2,\dotsc,n_m;n_1)$, then 
\begin{align*}
V:=\Big(\frac{1}{1+X_2+\dotsb+X_m},\frac{X_2}{1+X_2+\dotsb+X_m},\dotsc,\frac{X_m}{1+X_2+\dotsb+X_m}\Big)
\end{align*}
 is Dirichlet distributed with parameter set $(n_1,\dotsc,n_m)$.
\end{lemma}

\begin{proof}
For a proof we refer to \cite[Theorem 2.1]{Yass74}, \cite[Section 2]{Tiao_Cuttman1965} or \cite[Section 2.1]{Ghorbel2009}.
\end{proof}

The second important fact in this section, Lemma \ref{lem:conjugate},
 should be compared to the conjugate relation between gamma and Poisson distributed random variables, which has also been used in \cite[Explanation below Corollary 1]{Hornik24}. A more general form of Lemma  \ref{lem:conjugate} can be found in \cite[Section 17.4, eq.~(17.24)]{JohnKotzBal94}, which goes back to \cite{Pearson1922}.

\begin{lemma}\label{thm:conjugate-poisson}
If $P\sim\operatorname{Pois}(\lambda)$ and $G\sim\operatorname{Gamma}(n,1)$, then
\begin{align*}
\mathbb{P}(P< n)=\mathbb{P}(P\le n-1)=\mathbb{P}(G>\lambda).
\end{align*}
\end{lemma}

\begin{proof} Use repeated partial integration on 
\(
\mathbb{P}(G>\lambda)=\int_\lambda^\infty \frac{x^{n-1}e^{-x}}{\Gamma(n)}dx\,.
\)
\detail{
\begin{align*}
\mathbb{P}(G>\lambda)
&=\int_{\lambda}^\infty \frac{ x^{n-1}e^{- x}}{\Gamma(n)}dx
=\frac{x^{n-1}}{\Gamma(n)} e^{-\lambda}+\int_{\lambda}^\infty \frac{ x^{n-2}e^{- x}}{\Gamma(n-1)}dx
=\P(P=n-1)+\mathbb{P}(G'>\lambda)\,,
\end{align*}
where \(G'\sim \operatorname{Gamma}(n-1,1)\).
}
\end{proof}

\begin{lemma}\label{lem:conjugate} Let \(m\ge 2\) be an integer, \(n_1,\ldots,n_m\in \N_{\ge 1}\) and 
\(p_1,\ldots,p_m\in (0,1)\) with  \(p_1+\ldots+p_m=1\).
Let $(A_2,\dotsc,A_m)$ have a negative multinomial distribution with parameters $(n_1;p_2,\dotsc,p_m)$ and $(X_2,\dotsc,X_m)\sim \operatorname{ID}(n_2,\dotsc,n_m;n_1)$. Then
\begin{align*}
\mathbb{P}(A_2\leq n_2-1,\dotsc,A_m\leq n_m-1)=\mathbb{P}\left(X_2>\frac{p_2}{p_1},\dotsc,X_m>\frac{p_m}{p_1}\right).
\end{align*}
\end{lemma}

\begin{proof}
\begin{newenv}
Using Lemma \ref{lem:gammarep}, we rewrite
\begin{align*}
&\mathbb{P}\left(X_2>\frac{p_2}{p_1},\dotsc,X_m>\frac{p_m}{p_1}\right)=\mathbb{P}\left(G_2>\frac{p_2G_1}{p_1},\dotsc,G_m>\frac{p_mG_1}{p_1}\right)\\
&=\int_0^\infty\mathbb{P}\left(G_2>\frac{p_2x}{p_1},\dotsc,G_m>\frac{p_m x}{p_1}\right)\frac{x^{n_1-1}}{\Gamma(n_1)}e^{-x}dx,
\end{align*}
which equals, by Lemma \ref{thm:conjugate-poisson}, 
\begin{align}\label{eq:gamma-nmn}
\int_0^\infty\mathbb{P}\left(P_{2,x}<n_2,\dotsc,P_{m,x}<n_m\right)\frac{x^{n_1-1}}{\Gamma(n_1)}e^{-x}dx,
\end{align}
for an independent vector $(P_{2,x},\dotsc,P_{m,x})$ of Poisson random variables with parameter $\frac{p_kx}{p_1}$, $k=2,\dotsc,m$. By \cite[Section 3.b.]{Sibuya64}, the expression \eqref{eq:gamma-nmn} is the desired probability.
\end{newenv}
\end{proof}
\new{A different, quite direct approach to the proof is to iteratively use partial integration on the tail probability 
\[\mathbb{P}\left(X_2>\frac{p_2}{p_1},\dotsc,X_m>\frac{p_m}{p_1}\right)\] in all variables. \detail{
\begin{align*}
\frac{1}{\Beta(n_1,\ldots,n_m)}\int_{\frac{p_2}{p_1}}^\infty\frac{\prod_{k=2}^m s_k^{n_{k}-1}}{\Big(1+\sum_{k=2}^m s_k\Big)^{\sum_{k=1}^m n_k}}ds_2
&=\frac{1}{\Beta(n_1,\ldots,n_m)}\frac{1}{\sum_{k=1}^m n_k-1}\frac{\big(\frac{p_2}{p_1}\big)^{n_2-1}\prod_{k=3}^m s_k^{n_{k}-1}}{\Big(1+\frac{p_2}{p_1}+\sum_{k=3}^m s_k\Big)^{\sum_{k=1}^m n_k-1}}\\
&\quad+\frac{1}{\Beta(n_1,\ldots,n_m)}\frac{n_2-1}{\sum_{k=1}^m n_k-1}\int_{\frac{p_2}{p_1}}^\infty\frac{s_2^{n_2-2}\prod_{k=3}^m s_k^{n_{k}-1}}{\Big(1+\sum_{k=2}^m s_k\Big)^{\sum_{k=1}^m n_k-1}}ds_2
\end{align*}
}%
}

\begin{remark}
\begin{enumerate}
\item With the notation from Lemma \ref{lem:gammarep} we can also write
\begin{align*}
\mathbb{P}(A_2\leq n_2-1,\dotsc,A_m\leq n_m-1)=\mathbb{P}\left(V_2>p_2,\dotsc,V_m>p_m\right).
\end{align*}
\item \label{it:conjugate-rem}
More conjugacy relations like those in Lemma \ref{thm:conjugate-poisson} and Lemma \ref{lem:gammarep} can be established between pairs of 
 probability distributions (appearing e.g.~in the context of conjugate priors, see \cite[Part I, Chapter 3]{RaiffaSchleifer} or \cite[Section 5.2]{Smith00}). A further concrete example is given by the uniform distribution $\mathcal{U}_{[0,b]}$ and the Pareto distribution with scale parameter $x_m$ (and the fixed shape $\alpha=1$).
\end{enumerate}
\end{remark}

\begin{newenv} 
Before answering Question \ref{question1} in the next section, we present a duality result: Instead of prescribing the advancing probabilities \(p_1,\ldots,p_m\) and ask for the winning probabilities  \(\pi_1,\ldots,\pi_m\),
one can prescribe the latter and ask for the former. That this is indeed always possible, and that the correspondence is a diffeomorphism, is the content
of the next theorem.

\begin{theorem}\label{thm:pi-bij}
Let \(\Delta_{m-1}\) denote the \(m-1\) dimensional simplex in \(\R^m\).
Then the function $\pi\colon \Delta_{m-1}\to\Delta_{m-1}, p=(p_1,\dotsc,p_m)\mapsto(\pi_1(p),\dotsc,\pi_m(p))$ is a bijection.
The restriction of  $\pi$ to the relative interior of \(\Delta_{m-1}\) is a diffeomorphism. 
\end{theorem}

\begin{proof}
By Proposition \ref{thm:winning-prob-neg}, Lemma \ref{lem:conjugate}, and the definition of the inverted Dirichlet distribution, 
$$\pi_\ell(p) = \frac{1}{\Beta(n_1,\dotsc,n_m)}\int_{\frac{p_1}{p_\ell}}^\infty\cdots\int_{\frac{p_{\ell-1}}{p_\ell}}^\infty\int_{\frac{p_{\ell+1}}{p_\ell}}^\infty\cdots\int_{\frac{p_m}{p_\ell}}^\infty\frac{\prod_{k\neq \ell}s^{n_k-1}}{\Big(1+\sum_{k\neq \ell}^ms_k\Big)^{\sum_{k=1}^mn_k}}ds_1\cdots ds_m\,.$$
We want to rewrite this in a form more suitable for differentiation:
Consider the function $\Psi:\Delta_{m-1}\to[0,\infty]^{m-1}, p\mapsto\Big(\frac{p_2}{p_1},\dotsc,\frac{p_m}{p_1}\Big)$, which is a bijection and  is even diffeomorphic between the respective interiors. Its inverse is given by 
\[\Psi^{-1}(x_2,\ldots,x_m)=\left(\frac{1}{1+x_2+\ldots+x_m},\frac{x_2}{1+x_2+\ldots+x_m},\ldots,\frac{x_m}{1+x_2+\ldots+x_m}\right).\]

Define \(\tilde {\pi}\colon [0,\infty]^{m-1}\to \Delta_{m-1}\) by 
\(
\tilde {\pi}:=\pi\circ \Psi^{-1}.
\)
Then \(\Psi\circ \tilde\pi\) is a function \([0,\infty]^{m-1}\to [0,\infty]^{m-1}\). 
From the Hadamard global inverse function theorem (see \cite[Theorem 2.2]{ruzhansky2015}), it follows that if   \(\Psi\circ \tilde\pi\)
 has invertible Jacobian and maps the boundary of its domain into itself, then it is a diffeomorphism of the interior of its domain.

But since \(\Psi\) is a diffeomorphism, \(\Psi\circ \tilde\pi\) has invertible Jacobian if and only if \(\tilde \pi\) has, and  
\(\pi\) is a diffeomorphism if and only if \(\Psi\circ \tilde\pi\) is. Moreover, \(\Psi\circ \tilde\pi\) maps the boundary of its domain to the boundary,
since this is true for \(\pi\).

Thus it remains to show that $\tilde{\pi}$ has invertible Jacobian everywhere in the interior of its domain. 
For $2\leq \ell\leq m$ we get
\begin{align*}\scriptstyle
\tilde{\pi}_\ell(x) = \frac{1}{\Beta(n_1,\dotsc,n_m)}\int_{\frac{1}{x_\ell}}^\infty\int_{\frac{x_2}{x_\ell}}^\infty\cdots\int_{\frac{x_{\ell-1}}{x_\ell}}^\infty\int_{\frac{x_{\ell+1}}{x_\ell}}^\infty\cdots\int_{\frac{x_m}{x_\ell}}^\infty\frac{\prod_{k\neq \ell}s_k^{n_k-1}}{\Big(1+\sum_{k\neq \ell}^ms_k\Big)^{\sum_{k=1}^mn_k}}ds_1\cdots ds_{\ell-1}ds_{\ell+1}\cdots ds_m\,.
\end{align*}
Taking partial derivatives gives
\begin{align*}
\frac{\partial\tilde{\pi}_\ell(x)}{\partial x_\ell}=\sum_{\substack{1\leq j\leq m\\j\neq \ell}}\frac{x_j}{x_\ell^2}\cdot\frac{\Big(\frac{x_j}{x_\ell}\Big)^{n_j-1}}{\Big(1+\frac{x_j}{x_\ell}\Big)^{n_j+n_\ell}}\cdot I(j,\ell), \text{ and } \frac{\partial\tilde{\pi}_\ell(x)}{\partial x_k}=-\frac{1}{x_\ell}\cdot\frac{\Big(\frac{x_k}{x_\ell}\Big)^{n_k-1}}{\Big(1+\frac{x_k}{x_\ell}\Big)^{n_k+n_\ell}}\cdot I(k,\ell),
\end{align*}
for $2\leq k\leq m, k\neq \ell$, where $x_1:=1$ and where for \(1\le j\le m\)
\begin{align*}
 I(j,\ell):=\frac{1}{\Beta(n_1,\dotsc,n_m)}\!\underbrace{\int_{\frac{x_1}{x_\ell}}^\infty\cdots\int_{\frac{x_m}{x_\ell}}^\infty}_{\text{except }\frac{x_j}{x_\ell},\frac{x_\ell}{x_\ell}}\!\frac{\prod_{i\notin\{j,\ell\}}t_i^{n_i-1}}{\Big(1+\sum_{i\notin\{j,\ell\}}t_i\Big)^{\sum_{i=1}^mn_i}}\underbrace{dt_1\cdots dt_m}_{\text{except }dt_j,dt_\ell}\,.
\end{align*}
Therefore, in the interior of $[0,\infty]^{m-1}$, the Jacobian $J\tilde{\pi}(x)$ has a positive diagonal and all other entries are negative. It is therefore a Z-matrix \cite[Chapter 6 and 10]{BerPlem87}. Moreover, multiplying the Jacobian with the (positive) column vector $x^T$, we obtain for $2\leq \ell\leq m$,
\begin{align*}
\big(J\tilde{\pi}(x)x^T\big)_\ell=\sum_{k=2}^m x_k\frac{\partial\tilde{\pi}_\ell(x)}{\partial x_k} = \frac{1}{x_\ell}\cdot\frac{\Big(\frac{1}{x_\ell}\Big)^{n_\ell-1}}{\Big(1+\frac{1}{x_\ell}\Big)^{n_1+n_\ell}}\cdot I(1,\ell)>0\,.
\end{align*}
Therefore, $J\tilde{\pi}(x)x^T$ has positive entries, and so, following results on properties of Z-matrices \cite[Chapter 6, criterion I$_{27}$]{BerPlem87} or \cite[criterion K$_{33}$]{Plemmons77}, $J\tilde{\pi}$ is invertible in $(0,\infty)^{m-1}$.

Finally, that \(\pi\) is also a bijection on the faces of $\Delta_{m-1}$
follows from the same arguments and Remark \ref{rem:idir}\eqref{it:idir-marg}, since setting 
\(p_\ell=0\) implies \(\pi_\ell=0\).
\end{proof}

In particular, Theorem \ref{thm:pi-bij} grants the existence and uniqueness of an `equal probability vector' $\bar{p}\in \Delta_{m-1}$ with 
$\pi(\bar p)=(\tfrac{1}{m},\dotsc,\tfrac{1}{m})$ for all $m\geq 2$. 

For the case \(m=2\), we can give estimates on the components of \(\bar p\) using \cite[Theorem 3.7]{goeb}: 
For \(1=n_1\le n_2\) we get the exact value \(\bar p_1=1-2^{-\frac{1}{n_2}}\),
for \(2\le n_1\le n_2\) we get the bounds 
\(\frac{n_1-1}{n_1+n_2-1}< \bar p_1 <\frac{n_1}{n_1+n_2}\), and accordingly, \(\frac{n_2}{n_1+n_2}< \bar p_2 <\frac{n_2}{n_1+n_2-1}\). 
Note that for the bounds on \(\bar p_1\) we have the following:
\begin{itemize}
\item If we set \(p_1(n_1,n_2)=\frac{n_1}{n_1+n_2}\), then the function \(n_1\mapsto \pi_1\big(p_1(n_1,n_2),1-p_1(n_1,n_2)\big)\) is decreasing on \(\N_{\ge 1}\), by our results in Section \ref{sec:twoplayers}.
\item If we set \(p_1(n_1,n_2)=\frac{n_1-1}{n_1+n_2-1}\), then the function \(n_1\mapsto \pi_1\big(p_1(n_1,n_2),1-p_1(n_1,n_2)\big)\) is increasing on \(\N_{\ge 1}\),
which can be seen from \cite[Theorem 2]{Hornik24}.
\end{itemize} 
\end{newenv}

For most of the remainder of this article we will specialize to the setting \(p_\ell:=\frac{n_\ell}{n_1+\ldots+n_m}\), and we 
regard the winning probabilities \(\pi_1,\ldots,\pi_m\) as functions of \(n_1,\ldots,n_m\).

\section{Winning the game}\label{sec:winning}

We formulate a theorem summarizing the results of Section \ref{sec:gprob} for the purpose of studying the class of games described in the
introduction. 

\begin{theorem}\label{thm:rep-pi1}
Let \(m\ge 2\).  For all \(n_1,\ldots,n_m\in \N_{\ge 1}\) consider
\[
\pi_1(n_1,\ldots,n_m)=\mathbb{P}(A_2\leq n_2-1,\dotsc,A_m\leq n_m-1) 
\]
where $(A_2,\dotsc,A_m)$ has a negative multinomial distribution with parameters $(n_1;p_2,\dotsc,p_m)$, 
where \(p_\ell=\frac{n_\ell}{n_1+\ldots,n_m}\) for \(\ell\in\{2,\ldots,m\}\).

Then we have the following representations:
\begin{enumerate}
\item \label{it:rep-pi1-1}
\(\pi_1(n_1,\ldots,n_m)=\mathbb{P}\left(X_2>\frac{n_2}{n_1},\dotsc,X_m>\frac{n_m}{n_1}\right)\)\\
where $(X_2,\dotsc,X_m)\sim \operatorname{ID}(n_2,\dotsc,n_m;n_1)$.
\item \label{it:rep-pi1-2}\(\pi_1(n_1,\ldots,n_m)=\mathbb{P}\left(\frac{G_2}{n_2} > \frac{G_1}{n_1},\dotsc,\frac{G_m}{n_m} > \frac{G_1}{n_1}\right)\)\\
where \(G_1,\ldots,G_m\) are independent random variables with \(G_j\sim \operatorname{Gamma}(n_j,1)\), \(j=1,\ldots,m\).
\end{enumerate}
\end{theorem}

\begin{proof} (1) follows immediately from Lemma \ref{lem:conjugate}, (2) follows from (1) and Lemma \ref{lem:gammarep}.\end{proof}

We are ready to answer Question \ref{question1}, that is, 
we will show that, given goals \(n_1,\ldots,n_m\) with respective advancement probabilities 
\(\frac{n_1}{n_1+\cdots+n_m},\ldots,\frac{n_m}{n_1+\cdots+n_m}\), the players with the lowest goal have the highest winning probabilities.
This follows from the stronger result that, for a given player, the probability of winning  increases if any of the other player's goal is increased.
By symmetry, it is enough to show this for player 1. 

\begin{theorem}
The function \(\pi_1\) defined in Theorem \ref{thm:rep-pi1} is increasing in \(n_2,\ldots,n_m\).
\end{theorem}  

\begin{proof}
In order to show that the winning probability \(\pi_1(n_1,\ldots,n_m)\) of player $1$ increases if either of the $n_2,\dotsc,n_m$ increases, it suffices to show this monotonicity for $n_m$ only. For the other arguments this follows by symmetry. 
By Theorem \ref{thm:rep-pi1}
\begin{align*}
\pi_1&=\mathbb{P}\left(\frac{G_2}{n_2} > \frac{G_1}{n_1},\dotsc,\frac{G_m}{n_m} > \frac{G_1}{n_1}\right),
\end{align*}
where $(G_1,\dotsc,G_m)$ is a vector of independent random variables with $G_k\sim\operatorname{Gamma}(n_k,1)$ 
and where we
suppressed the arguments of \(\pi_1\) for brevity. 
Using disintegration for \(G_1\) and independence of \(G_2,\ldots,G_m\)  yields
\begin{align*}
\pi_1
&=\int_{0}^\infty \mathbb {P}\left(\frac{G_2}{n_2} > \frac{s}{n_1},\dotsc,\frac{G_m}{n_m} > \frac{s}{n_1}\right)\frac{s^{n_1-1}e^{-s}}{\Gamma(n_1)}ds\\
&=\int_{0}^\infty \prod_{k=2}^m\mathbb{P}\left(\frac{G_k}{n_k} > \frac{s}{n_1}\right)\frac{s^{n_1-1}e^{-s}}{\Gamma(n_1)}ds.
\end{align*}
Change of variables \(x=\frac{s}{n_1}\) gives 
\begin{align}
\label{eq:Pdisint}
\pi_1
&=\int_{0}^\infty \prod_{k=2}^m\mathbb{P}\left(\frac{G_k}{n_k} >  x\right)n_1^{n_1}\frac{x^{n_1-1}e^{-n_1 x}}{\Gamma(n_1)}dx.
\end{align}
Next, using Lemma \ref{thm:conjugate-poisson}, we express  for $2\leq \ell\leq m-1$ the probabilities $\mathbb{P}\left(\frac{G_\ell}{n_\ell} >   x\right)=\mathbb{P}\left(G_\ell >  n_\ell x\right)$ of the $\operatorname{Gamma}(n_\ell,1)$-distributed random variables $G_\ell$ with the help of  conjugate $\operatorname{Pois}(n_\ell x)$-distributed random variables $P_\ell$, that is,
\begin{align*}
\mathbb{P}\left(G_\ell >  n_\ell x\right)=\mathbb{P}(P_\ell<n_\ell)=e^{-n_\ell x}\sum_{k_\ell=0}^{n_\ell-1}\frac{(n_\ell x)^{k_\ell}}{k_\ell!}.
\end{align*}
Inserting this into \eqref{eq:Pdisint} and pulling constants out of the integral, we get for the winning probability of the first player,
\begin{align*}
\pi_1&=\frac{n_1^{n_1}}{\Gamma(n_1)}\sum_{k_2=0}^{n_2-1}\cdots\!\!\!\sum_{k_{m-1}=0}^{n_{m-1}-1}\Big(\prod_{\ell=2}^{m-1}\frac{ n_\ell^{k_\ell}}{k_\ell!}\Big)\int_0^\infty\mathbb{P}\left(G_m >  n_mx\right)x^{n_1+\sum_{\ell=2}^{m-1} k_\ell-1}e^{-(\sum_{\ell=1}^{m-1} n_\ell)x}dx.
\end{align*}
Next, we set \(k_1:=n_1\), such that the sum \(n_1+\sum_{\ell=2}^{m-1} k_\ell\) abbreviates to \(\sum_{\ell=1}^{m-1} k_\ell\), and we use change of variables to write
\begin{align}
\nonumber\pi_1&=\frac{n_1^{n_1}}{\Gamma(n_1)}\sum_{k_2=0}^{n_2-1}\cdots\!\!\!\sum_{k_{m-1}=0}^{n_{m-1}-1}\Big(\prod_{\ell=2}^{m-1}\frac{ n_\ell^{k_\ell}}{k_\ell!}\Big)\int_0^\infty\mathbb{P}\left(G_m >  \frac{n_my}{\sum_{\ell=1}^{m-1} n_\ell}\right)\frac{y^{\sum_{\ell=1}^{m-1} k_\ell-1}}{(\sum_{\ell=1}^{m-1} n_\ell)^{\sum_{\ell=1}^{m-1} k_\ell}}e^{-y}dy\\
\label{eqn:winning-sum-prod}&=\frac{n_1^{n_1}}{\Gamma(n_1)}\sum_{k_2=0}^{n_2-1}\cdots\!\!\!\sum_{k_{m-1}=0}^{n_{m-1}-1}\Big(\prod_{\ell=2}^{m-1}\frac{ n_\ell^{k_\ell}}{k_\ell!}\Big)\frac{\Gamma\Big(\sum_{\ell=1}^{m-1} k_\ell\Big)}{\Big(\sum_{\ell=1}^{m-1} n_\ell\Big)^{\sum_{\ell=1}^{m-1} k_\ell}}\mathbb{P}\left(G_m >  \frac{n_mH_k}{\sum_{\ell=1}^{m-1} n_\ell}\right),
\end{align}
where $H_k\sim\operatorname{Gamma}\Big(\sum_{\ell=1}^{m-1} k_\ell,1\Big)$, \(k=(k_1,\ldots,k_{m-1})\), 
with \(k_\ell\in \{0,\ldots, n_\ell-1\}\) and 
every $H_k$ is independent of $G_m$. Hence the quotient $\frac{G_m}{H_k}$ is $\operatorname{ID}\Big(n_m;\sum_{\ell=1}^{m-1} k_\ell\Big)$-distributed and we calculate 
 \begin{align}
 \nonumber&\mathbb{P}\left(G_m >  \frac{n_mH_k}{\sum_{\ell=1}^{m-1} n_\ell}\right)
 =\mathbb{P}\left(\frac{G_m}{H_k} >  \frac{n_m}{\sum_{\ell=1}^{m-1} n_\ell}\right)\\
 \nonumber &=\frac{\Gamma\Big(\sum_{\ell=1}^{m-1} k_\ell+n_m\Big)}{\Gamma\Big(\sum_{\ell=1}^{m-1} k_\ell\Big)\Gamma(n_m)}\int_{\frac{n_m}{\sum_{\ell=1}^{m-1} n_\ell}}^\infty\frac{s^{n_m-1}}{(1+s)^{\sum_{\ell=1}^{m-1} k_\ell+n_m}}ds \\
\label{eqn:pre-lemma-beta} &= I_{\frac{1}{1+\frac{n_m}{\sum_{\ell=1}^{m-1} n_\ell}}}\Big(\sum_{\ell=1}^{m-1} k_\ell,n_m\Big)= I_{\frac{\sum_{\ell=1}^{m-1} n_\ell}{\sum_{\ell=1}^m n_\ell}}\Big(\sum_{\ell=1}^{m-1} k_\ell,n_m\Big),
 \end{align}
 where $I$ denotes the incomplete regularized beta function. If we can show that this expression increases in \(n_m\), then
  it follows that \(\pi_1\) is increasing (cf.~\eqref{eqn:winning-sum-prod}).
 The fact that the last term in \eqref{eqn:pre-lemma-beta}  increases in \(n_m\) follows from the subsequent Lemma \ref{th:inc-reg-beta-dec} by setting \(N=\sum_{\ell=2}^{m-1} n_\ell\) and \(K=\sum_{\ell=2}^{m-1} k_\ell\). 
\end{proof}

\begin{lemma}\label{th:inc-reg-beta-dec}
Let $n_1\in (0,\infty)$  and let $0\leq K\leq N$ be real numbers.  Then the function \(\N_{\ge 1}\to \R\)
\begin{align*}
n_m\mapsto I_{\frac{n_1+N}{n_1+N+n_m}}(n_1+K,n_m)
\end{align*}
is strictly increasing.
\end{lemma}
\begin{proof} We want to show that for a given \(n_m\in \N_{\ge 1}\)
\[
I_{\frac{n_1+N}{n_1+N+n_m+1}}(n_1+K,n_m+1)>I_{\frac{n_1+N}{n_1+N+n_m}}(n_1+K,n_m).
\]
For this, we use the representation 
\[
I_{\frac{n_1+N}{n_1+N+n_m}}(n_1+K,n_m)= \frac{\Gamma(n_1+K+n_m)}{\Gamma(n_1+K)\Gamma(n_m)}\int_{\frac{n_m}{n_1+N}}^\infty\frac{s^{n_m-1}}{(1+s)^{n_1+K+n_m}}ds.
\]
After cancelling the terms involving the gamma function, 
we need to check that
\begin{align}
\label{eq:ineq-lemma}\frac{n_1+K+n_m}{n_m}\int_{\frac{n_m+1}{n_1+N}}^\infty\frac{s^{n_m}}{(1+s)^{n_1+K+n_m+1}}ds
>\int_{\frac{n_m}{n_1+N}}^\infty\frac{s^{n_m-1}}{(1+s)^{n_1+K+n_m}}ds.
\end{align}
Using partial integration on the right hand side, 
\detail{
\begin{align*}
\int_{\frac{n_m}{n_1+N}}^\infty s^{n_m-1}\frac{1}{(1+s)^{n_1+K+n_m}}ds
&=\frac{s^{n_m}}{n_m}\frac{1}{(1+s)^{n_1+K+n_m}}\Big|_{\frac{n_m}{n_1+N}}^\infty
+\int_{\frac{n_m}{n_1+N}}^\infty \frac{s^{n_m}}{n_m}\frac{n_1+K+n_m}{(1+s)^{n_1+K+n_m+1}}ds
\end{align*}
}
we get
\begin{align*}
&\frac{n_1+K+n_m}{n_m}\int_{\frac{n_m+1}{n_1+N}}^\infty\frac{s^{n_m}}{(1+s)^{n_1+K+n_m+1}}ds\\
&>-\frac{\Big(\frac{n_m}{n_1+N}\Big)^{n_m}}{n_m\Big(1+\frac{n_m}{n_1+N}\Big)^{n_1+K+n_m}}+\frac{n_1+K+n_m}{n_m}\int_{\frac{n_m}{n_1+N}}^\infty\frac{s^{n_m}}{(1+s)^{n_1+K+n_m+1}}ds,
\end{align*}
which is equivalent to
\begin{align*}
\frac{\Big(\frac{n_m}{n_1+N}\Big)^{n_m}}{n_m\Big(1+\frac{n_m}{n_1+N}\Big)^{n_1+K+n_m}}>\frac{n_1+K+n_m}{n_m}\int_{\frac{n_m}{n_1+N}}^{\frac{n_m+1}{n_1+N}}\frac{s^{n_m}}{(1+s)^{n_1+K+n_m+1}}ds.
\end{align*}
The right hand side can always be strictly estimated from above by 
\begin{align*}
\frac{n_1+K+n_m}{n_m}\cdot\frac{1}{n_1+N}\cdot\frac{\Big(\frac{n_m+h}{n_1+N}\Big)^{n_m}}{\Big(1+\frac{n_m+h}{n_1+N}\Big)^{n_1+K+n_m+1}},
\end{align*}
with some $0\leq h\leq 1$ depending on all other variables. \detail{Take the unique \[
h:=\operatorname{argmax}_{h'\in [0,1]} \frac{\Big(\frac{n_m+h'}{n_1+N}\Big)^{n_m}}{\Big(1+\frac{n_m+h'}{n_1+N}\Big)^{n_1+K+n_m+1}}
\] and note that the integrand is continuous and not constant.}
Hence, to show \eqref{eq:ineq-lemma}, it is sufficient to prove that for all $h\in [0,1]$,
\begin{align*}
\frac{\Big(\frac{n_m}{n_1+N}\Big)^{n_m}}{n_m\Big(1+\frac{n_m}{n_1+N}\Big)^{n_1+K+n_m}}\geq \frac{n_1+K+n_m}{n_m(n_1+N)}\cdot\frac{\Big(\frac{n_m+h}{n_1+N}\Big)^{n_m}}{\Big(1+\frac{n_m+h}{n_1+N}\Big)^{n_1+K+n_m+1}},
\end{align*}
or, simplified,
\begin{align*}
\frac{n_m^{n_m}}{\Big(\frac{n_1+N+n_m}{n_1+N}\Big)^{n_1+K+n_m}}\geq \frac{n_1+K+n_m}{n_1+N}\cdot\frac{(n_m+h)^{n_m}}{\Big(\frac{n_1+N+n_m+h}{n_1+N}\Big)^{n_1+K+n_m+1}},
\end{align*}
\detail{that is,
\begin{align*}
\Big(\frac{n_1+N+n_m+h}{n_1+N+n_m}\Big)^{n_1+K+n_m}\cdot\frac{n_1+N+n_m+h}{n_1+K+n_m}\geq \Big(\frac{n_m+h}{n_m}\Big)^{n_m},
\end{align*}}
which is equivalent to
\begin{align}\label{eq:k-basic_m}
\Big(1+\frac{h}{n_1+N+n_m}\Big)^{n_1+K+n_m}\cdot\frac{n_1+N+n_m+h}{n_1+K+n_m}\geq \Big(1+\frac{h}{n_m}\Big)^{n_m}.
\end{align}
Note that the left hand side is non-increasing in \(K\): differentiating w.r.t.~\(K\) yields
\begin{align*}
&\log\Big(1+\frac{h}{n_1+N+n_m}\Big)\Big(1+\frac{h}{n_1+N+n_m}\Big)^{n_1+K+n_m}\cdot\frac{n_1+N+n_m+h}{n_1+K+n_m}\\
&\quad-\Big(1+\frac{h}{n_1+N+n_m}\Big)^{n_1+K+n_m}\cdot\frac{n_1+N+n_m+h}{(n_1+K+n_m)^2}.
\end{align*}
Since $$\log\Big(1+\frac{h}{n_1+N+n_m}\Big)\leq \frac{h}{n_1+N+n_m}\leq \frac{1}{n_1+N+n_m}\leq \frac{1}{n_1+K+n_m},$$
the derivative is non-positive. Therefore, we only need to show \eqref{eq:k-basic_m} for $K=N$, i.e.,
\begin{align*}
&\Big(1+\frac{h}{n_1+N+n_m}\Big)^{n_1+N+n_m}\cdot\frac{n_1+N+n_m+h}{n_1+N+n_m}\geq \Big(1+\frac{h}{n_m}\Big)^{n_m}
\end{align*}
which holds since $r\mapsto\Big(1+\frac{h}{r}\Big)^r$ is non-decreasing for \(h\ge 0\).
\end{proof}

\begin{remark}
The same calculation as in the proof of Lemma \ref{th:inc-reg-beta-dec} shows that 
for $n_1\ge 0$, \(0< K\le N\)  the function \(\N_{\ge 1}\to \R\)
\begin{align*}
n_m\mapsto \frac{\Gamma(n_1+K+n_m)}{\Gamma(n_1+K)\Gamma(n_m)}\int_{\frac{n_m}{n_1+N}}^\infty\frac{s^{n_m-1}}{(1+s)^{n_1+K+n_m}}ds=I_{\frac{n_1+N}{n_1+N+n_m}}(n_1+K,n_m)
\end{align*}
is strictly increasing. That is, if we require \(0< K\le N\), then the conclusion of Lemma \ref{th:inc-reg-beta-dec} also holds for \(n_1=0\).
\end{remark}

\begin{remark} 
As mentioned in Section \ref{sec:twoplayers}, Lemma \ref{th:inc-reg-beta-dec} is closely related to \cite[Theorem 1]{Hornik24},
which is more special in the sense that \(K=N=0\) and is more general in the sense that it allows for non-integer \(n_m\).

Note also that our proof does not employ deep results about incomplete beta functions.  
\end{remark}

\section{Asymptotics}\label{sec:asymptotics}

We now study the asymptotic behavior of the winning probability \(\pi_1\) with respect to 
\(n_1,\ldots, n_m\). For later use we state two results about limits of gamma random variables.

\begin{lemma}\label{thm:gamma-conv}
Let \((G_n)_{n\in \N}\) be a sequence of random variables, \(G_n\sim \operatorname{Gamma}(n,1)\).
Then 
\[
\lim_{n\to\infty}\frac{G_n}{n}= 1 \quad\text{in probability.}
\]
\end{lemma}  

\begin{proof} Let \((Y_k)_{k\in \N}\) be a sequence of independent $\operatorname{Exp}(1)$-distributed random variables. Let \(S_n=Y_1+\ldots+Y_n\). 
Then \(G_n\stackrel {d}{=}S_n\).
By the strong law of large numbers we have \(\P(\lim_{n\to \infty } \tfrac{S_n}{n}=1)=1\). From this we get 
\(\lim_{n\to \infty } \tfrac{S_n}{n}=1\) in distribution, and hence \(\lim_{n\to \infty } \tfrac{G_n}{n}=1\) in distribution. Since the limit is a constant, 
convergence in probability follows by \cite[Theorem 25.3]{billingsley1986probability}.
\end{proof}

\begin{proposition}\label{eq:gamma-clt}
Let \(\alpha>0\) and let \((G_n)_{n\in \N_{\ge 1}}\) be a sequence of random variables, where \(G_n\) has a \(\operatorname{Gamma}(\alpha n,1)\) distribution, for every \(n\in\N_{\ge 1}\). Then 
\[
\lim_{n\to \infty}\P\Big(\sqrt{\alpha n}\Big(\frac{G_n}{\alpha n}-1\Big)\le b\Big)=\Phi(b),
\]
where \(\Phi(b)=\int_{-\infty}^b\frac{1}{\sqrt{2\pi}}e^{-\frac{z}{2}}dz\) denotes the standard normal CDF.
\end{proposition}

\begin{proof}
See \new{\cite[Section 17.2]{JohnKotzBal94}} or use the central limit theorem and the fact that
a gamma random variable essentially is the sum of independent exponential random variables.
\end{proof}

\begin{theorem}[Winning probability for \(n_1\to\infty\)]\label{thm:n1-infty} Let \(\pi_1\) be the function defined in Theorem \ref{thm:rep-pi1}.
Then 
\[
\lim_{n_1\to \infty}\pi_1(n_1,\ldots,n_m)=\prod_{k=2}^m\mathbb{P}(G_k >  n_k)
\]
where is  $G_\ell$ is $\operatorname{Gamma}(n_\ell,1)$-distributed, \(\ell=2,\ldots,m\).
\end{theorem}

\begin{proof} 
We use the representation of \(\pi_1\) from Theorem \ref{thm:rep-pi1}\eqref{it:rep-pi1-2},
\[\pi_1(n_1,\ldots,n_m)=\mathbb{P}\left(\frac{G_2}{n_2}>\frac{G_1}{n_1},\dotsc,\frac{G_m}{n_m}>\frac{G_1}{n_1}\right)\]
where $(G_1,\dotsc,G_m)$ is a vector of independent random variables $G_\ell$, which are $\operatorname{Gamma}(n_\ell,1)$-distributed.
By Lemma \ref{thm:gamma-conv}
\begin{equation}\label{eq:pre-slutsky}
\lim_{n_1\to\infty}\frac{G_1}{n_1}\to 1 \quad\text{in probability,} 
\end{equation}
so invoking Slutsky's Theorem gives that 
\[
\left(\frac{G_2}{n_2}-\frac{G_1}{n_1},\dotsc,\frac{G_m}{n_m}-\frac{G_1}{n_1}\right)\stackrel{d}{\to}\left(\frac{G_2}{n_2}-1,\dotsc,\frac{G_m}{n_m}-1\right) .
\]  
Since the tuple $\left(\frac{G_2}{n_2}-1,\dotsc,\frac{G_m}{n_m}-1\right)$ has a continuous (multivariate) distribution,
\begin{align*}
&\lim_{n_1\to\infty}\mathbb{P}\left(\frac{G_2}{n_2}>\frac{G_1}{n_1},\dotsc,\frac{G_m}{n_m}>\frac{G_1}{n_1}\right)
=\mathbb{P}\left(\frac{G_2}{n_2}>1,\dotsc,\frac{G_m}{n_m}>1\right)
=\prod_{k=2}^m\mathbb{P}(G_k >  n_k),
\end{align*}
where we used the independence of \(G_2,\ldots,G_m\) in the last step. 
\end{proof}

\begin{remark}
For the proof of Theorem \ref{thm:n1-infty} one could also start with the representation of \(\pi_1\) from Theorem \ref{thm:rep-pi1}(i),
\[\pi_1(n_1,\ldots,n_m)=\mathbb{P}\left(X_2>\frac{n_2}{n_1},\dotsc,X_m>\frac{n_m}{n_1}\right)\]
where $(X_2,\dotsc,X_m)\sim \operatorname{ID}(n_2,\dotsc,n_m;n_1)$ and use 
\cite[Section 5]{Tiao_Cuttman1965}.  There the authors show via series expansions that 
\begin{align*}
(n_1X_2,\dotsc,n_1X_m)\stackrel{d}{\to}(G_2,\dotsc,G_m)\;(n_1\to \infty),
\end{align*}
where $(G_2,\dotsc,G_m)$ is a vector of independent random variables $G_\ell$, which are $\operatorname{Gamma}(n_\ell,1)$-distributed. 
\end{remark}

The special case $m=2$ of Theorem \ref{thm:n1-infty}, interpreted as property of the negative binomial distribution has been discovered in \cite{Pes61}.\smallskip

Using Theorem \ref{thm:n1-infty} and  Lemma \ref {thm:conjugate-poisson}, we get that the limiting probability can also be expressed as
$$\lim_{n_1\to \infty}\pi_1(n_1,\ldots,n_m)=\prod_{k=2}^m\mathbb{P}(G_k >  n_k)=\prod_{k=2}^m\mathbb{P}(P_k<n_k),$$
where $(P_2,\dotsc,P_m)$ is a vector of independent random variables with $P_i\sim\operatorname{Pois}(n_i)$.
Recall that a \(\operatorname{Pois}(n)\) random variable has the same distribution as the sum of \(n\) independent \(\operatorname{Pois}(1)\) random variables.
Thus, by the central limit theorem, for a sequence \(P_n\) of random variables with \(P_n\sim \operatorname{Pois}(n)\) we have
for all \(x\in \R\)
\[
\lim_{n\to \infty}\mathbb{P}\Big(\frac{P_n-n}{\sqrt{n}}<x\Big)=\int_{-\infty}^x \frac{1}{\sqrt{2\pi}}e^{-\frac{x^2}{2}}dx,
\]
and therefore \(
\lim_{n\to \infty}\mathbb{P}(P_n<n)=\lim_{n\to \infty}\mathbb{P}(\frac{P_n-n}{\sqrt{n}}<0)=\frac{1}{2}.
\) 
Hence we get the surprising fact 
\begin{equation}\label{eq:gstoert}
\lim_{n_m\to \infty}\ldots\lim_{n_1\to \infty}\pi_1(n_1,\ldots,n_m)=\frac{1}{2^{m-1}}.
\end{equation}

The next theorem shows that the picture looks different if we let one of the other $n_k$, say $n_m$, tend to \(\infty\) first.

\begin{theorem}[Winning probability for \(n_m\to\infty\)]\label{thm:nm-infty} With \(X_2,\ldots,X_m\) and \(\pi_1\) as in Theorem \ref{thm:rep-pi1}(i),  
\begin{align*}
&\lim_{n_m\to \infty}\pi_1(n_1,n_2,\ldots,n_m)\\
&=\int_{\frac{n_2}{n_1}}^\infty\cdots\int_{\frac{n_{m-1}}{n_1}}^\infty\int_{0}^{n_1(1+\sum_{k=2}^{m-1}s_k)}f_\Gamma(s_m)ds_{m}\ \varphi_{<m}(s_2,\ldots,s_{m-1})ds_{m-1}\dotsc ds_2,
\end{align*}
where 
\(\varphi_{<m}\) is the marginal density of \(X_2,\ldots,X_{m-1}\) and \(f_\Gamma\) is a \(\operatorname{Gamma}\Big(\sum_{k=1}^{m-1}n_k,1\Big)\) density.
\end{theorem}

\begin{proof}
We omit the parameters of \(\pi_1\) for brevity.
Recall that by Theorem \ref{thm:rep-pi1}\eqref{it:rep-pi1-1} 
the winning probability of the first player is
\[
\pi_1=\frac{\Gamma\Big(\sum_{k=1}^m n_k\Big)}{\prod_{k=1}^m\Gamma(n_k)}\int_{n_2/n_1}^\infty\cdots\int_{n_m/n_1}^\infty \frac{\prod_{k=2}^m s_k^{n_k-1}}{\Big(1+\sum_{k=2}^ms_k\Big)^{\sum_{k=1}^mn_k}}ds_m\dotsc ds_2\,.
\]
\detail{Abbreviating \(A:=1+\sum_{k=2}^{m-1}s_k\), the integrand can be rewritten
\begin{align*}
\frac{\prod_{k=2}^m s_k^{n_k-1}}{\Big(1+\sum_{k=2}^ms_k\Big)^{\sum_{k=1}^mn_k}}
&=\frac{s_m^{n_m-1}\prod_{k=2}^{m-1} s_k^{n_k-1}}{\Big(1+s_m+\sum_{k=2}^{m-1}s_k\Big)^{\sum_{k=1}^mn_k}}
=\frac{\frac{1}{A}\big(\frac{s_m}{A}\big)^{n_m-1}\prod_{k=2}^{m-1} s_k^{n_k-1}}{\big(1+\frac{s_m}{A}\big)^{\sum_{k=1}^mn_k}\Big(1+\sum_{k=2}^{m-1}s_k\Big)^{\sum_{k=1}^{m-1}n_k}}.
\end{align*}
}
It can be rewritten via substitution  as
\[\scriptstyle
\pi_1=\frac{\Gamma\Big(\sum_{k=1}^m n_k\Big)}{\prod_{k=1}^m\Gamma(n_k)}\int_{\frac{n_2}{n_1}}^\infty\cdots\int_{\frac{n_{m-1}}{n_1}}^\infty\int_{\frac{n_m}{n_1(1+s_2+\dotsb+s_{m-1})}}^\infty \frac{s_m^{n_m-1}}{(1+s_m)^{\sum_{k=1}^m n_k}} ds_m \frac{\prod_{k=2}^{m-1} s_k^{n_k-1}}{\Big(1+\sum_{k=2}^{m-1}s_k\Big)^{\sum_{k=1}^{m-1}n_k}}ds_{m-1}\dotsc ds_2,
\]
which gives, distributing the gamma functions differently,
\begin{align*}
\pi_1=\int_{\frac{n_2}{n_1}}^\infty\cdots\int_{\frac{n_{m-1}}{n_1}}^\infty&\int_{\frac{n_{m}}{n_1(1+s_2+\dotsb+s_{m-1})}}^\infty \frac{\Gamma\Big(\sum_{k=1}^m n_k\Big)}{\Gamma(n_m)\Gamma\Big(\sum_{k= 1}^{m-1}n_k\Big)}\frac{s_m^{n_m-1}}{(1+s_m)^{\sum_{k=1}^m n_k}} ds_m\\
& \cdot\frac{\Gamma\Big(\sum_{k=1}^{m-1} n_k\Big)}{\prod_{k=1}^{m-1}\Gamma(n_k)}\frac{\prod_{k=2}^{m-1} s_k^{n_k-1}}{\Big(1+\sum_{k=2}^{m-1}s_k\Big)^{\sum_{k=1}^{m-1}n_k}}ds_{m-1}\dotsc ds_2.
\end{align*}

Note, that the integrand's second factor,
$$\frac{\Gamma\Big(\sum_{k=1}^{m-1} n_k\Big)}{\prod_{k=1}^{m-1}\Gamma(n_k)}\frac{\prod_{k=2}^{m-1} s_k^{n_k-1}}{\Big(1+\sum_{k=2}^{m-1}s_k\Big)^{\sum_{k=1}^{m-1}n_k}},$$
is the density of an $\operatorname{ID}(n_2,\dotsc,n_{m-1};n_1)$ distribution, coinciding with the marginal density of $(X_2,\dotsc,X_{m-1})$ cf.~Remark \ref{rem:idir}\eqref{it:idir-marg}. The first factor, 
$$\frac{\Gamma\Big(\sum_{k=1}^m n_k\Big)}{\Gamma(n_m)\Gamma\Big(\sum_{k=1}^{m-1}n_k\Big)} \frac{s_m^{n_m-1}}{(1+s_m)^{\sum_{k=1}^m n_k}} $$
is the density of a univariate inverted-Dirichlet distribution with parameter set $\Big(n_m;\sum_{k=1}^{m-1} n_k,1\Big)$ (or `inverted beta' or  `beta-prime' distribution). Since, by Lemma \ref{lem:gammarep}, such a distribution emerges from a quotient $\frac{G_m}{G_{\Sigma}}$ of independent random variables, where $G_m\sim\operatorname{Gamma}(n_m,1)$ and $G_{\Sigma}\sim\operatorname{Gamma}\Big(\sum_{k=1}^{m-1}n_k,1\Big)$, the integral
\begin{align*}
\int_{\frac{n_m}{n_1(1+s_2+\dotsb+s_{m-1})}}^\infty \frac{\Gamma\Big(\sum_{k=1}^m n_k\Big)}{\Gamma(n_m)\Gamma\Big(\sum_{k=1}^{m-1}n_k\Big)}\frac{s_m^{n_m-1}}{(1+s_m)^{\sum_{k=1}^m n_k}} ds_m
\end{align*}
is actually the probability
$$\mathbb{P}\Big(\frac{G_m}{G_{\Sigma}} >  \frac{n_m}{n_1(1+s_2+\dotsb+s_{m-1})}\Big)=\mathbb{P}\Big(\frac{G_m}{n_m} >  \frac{G_{\Sigma}}{n_1(1+s_2+\dotsb+s_{m-1})}\Big).$$
Since by Lemma \eqref{thm:gamma-conv}, 
\(
\lim_{n_m\to\infty}\frac{G_m}{n_m}\to 1 \;\text{in probability}, 
\)
we get that
\begin{align*}
&\lim_{n_m\to\infty}\mathbb{P}\Big(\frac{G_m}{n_m} >  \frac{G_{\Sigma}}{n_1(1+s_2+\dotsb+s_{m-1})}\Big)\\
&=\mathbb{P}\Big(1 >  \frac{G_{\Sigma}}{n_1(1+s_2+\dotsb+s_{m-1})}\Big)
=\mathbb{P}\big(G_{\Sigma}< n_1(1+s_2+\dotsb+s_{m-1})\big)
\end{align*}
with the same arguments as below \eqref{eq:pre-slutsky}.
Thus, our limiting probability of player 1 winning as $n_m\to\infty$  is given by
\begin{align*}
\int_{\frac{n_2}{n_1}}^\infty\cdots\int_{\frac{n_{m-1}}{n_1}}^\infty\mathbb{P}\Big(G_{\Sigma}< n_1\Big(1+\sum_{k=2}^{m-1}s_k\Big)\Big)\frac{\Gamma\Big(\sum_{k=1}^{m-1} n_k\Big)}{\prod_{k=1}^{m-1}\Gamma(n_k)}\frac{\prod_{k=2}^{m-1} s_k^{n_k-1}}{\Big(1+\sum_{k=2}^{m-1}s_k\Big)^{\sum_{k=1}^{m-1}n_i}}ds_{m-1}\dotsc ds_2,
\end{align*}
or, written out,
\begin{align*}\scriptstyle
\int_{\frac{n_2}{n_1}}^\infty\cdots\int_{\frac{n_{m-1}}{n_1}}^\infty\int_{0}^{n_1(1+\sum_{k=2}^{m-1}s_k)}\frac{s_m^{\sum_{k=1}^{m-1}n_k-1}e^{-s_m}}{\Gamma\Big(\sum_{k= 1}^{m-1}n_k\Big)}ds_m\frac{\Gamma\Big(\sum_{k=1}^{m-1} n_k\Big)}{\prod_{k=1}^{m-1}\Gamma(n_k)}\frac{\prod_{k=2}^{m-1} s_k^{n_k-1}}{\Big(1+\sum_{k=2}^{m-1}s_k\Big)^{\sum_{k=1}^{m-1}n_i}}ds_{m-1}\dotsc ds_2.
\end{align*}
\end{proof}

For several selected parameters (e.g.,~the last ones) $n_{\ell},\dotsc,n_m$ sent to infinity this works similarly: substituting $t_k=\frac{s_k}{1+s_{2}+\dotsb+s_{\ell-1}}$, \(k=\ell,\ldots,m\), the winning probability of the first player is
\begin{align*}
\int_{\frac{n_{2}}{n_1}}^\infty\dotsc\int_{\frac{n_{\ell-1}}{n_1}}^\infty &\int_{\frac{n_{\ell}}{n_1(1+s_{2}+\dotsb+s_{\ell-1})}}^\infty\cdots\int_{\frac{n_m}{n_1(1+s_{2}+\dotsb+s_{\ell-1})}}^\infty\varphi_{\ge\ell}(t_{\ell},\dotsc,t_m)dt_\ell\cdots dt_m\\
&\cdot \varphi_{<\ell}(s_{2},\dotsc,s_{\ell-1})ds_{2}\cdots ds_{\ell-1},
\end{align*}
where $\varphi_{\ge\ell}$ is the density of an inverted-Dirichlet distribution with parameter set $\Big(n_{\ell},\dotsc,n_m;\sum_{k=1}^{\ell-1} n_k\Big)$ and $\varphi_{<\ell}$ is the one with parameter set $(n_2,\dotsc,n_{\ell-1};n_1)$. Again, the integrals over $\varphi_{\ge \ell}$ can be represented as the probability
\begin{align*}
\mathbb{P}\left(\frac{G_\ell}{n_\ell} >  \frac{G_{\Sigma}}{n_1(1+s_{2}+\dotsb+s_{\ell-1})},\dotsc,\frac{G_m}{n_m} >  \frac{G_{\Sigma}}{n_1(1+s_{2}+\dotsb+s_{\ell-1})}\right),
\end{align*}
where $G_{k}\sim\operatorname{Gamma}(n_k,1), \ell\leq k\leq m$ and $G_\Sigma\sim \operatorname{Gamma}\Big(\sum_{k=1}^{\ell-1} n_k,1\Big)$ are independent.
Since, by Lemma \ref{thm:gamma-conv}, $\lim_{n_k\to\infty}\frac{G_k}{n_k}=1$ in probability for $\ell \leq k\leq m$, the limiting probability is
(following again the arguments below \eqref{eq:pre-slutsky})
\begin{align*}
&\mathbb{P}\big(G_\Sigma< n_1(1+s_{2}+\dotsb+s_{\ell-1}),\dotsc,G_\Sigma< n_1(1+s_{2}+\dotsb+s_{\ell-1})\big)\\
&=\mathbb{P}\big(G_\Sigma < n_1(1+s_{2}+\dotsb+s_{\ell-1})\big),
\end{align*}
so the winning probability of the first player, as $n_\ell,\dotsc,n_m$ tend to infinity, is
\begin{align*}
\int_{\frac{n_{2}}{n_1}}^\infty\dotsc\int_{\frac{n_{\ell-1}}{n_1}}^\infty \int_0^{n_1(1+s_{2}+\dotsb+s_{\ell-1})} \frac{x^{\sum_{k=1}^{\ell-1} n_k-1}e^{-x}}{\Gamma\Big(\sum_{k=1}^{\ell-1} n_k\Big)}dx\, \varphi_{<\ell}(s_2,\dotsc,s_{\ell-1})ds_{\ell-1}\cdots ds_{2},
\end{align*}
which is an integral over a multivariate Liouville distribution density (cf.~\cite[Section 6]{Fang_et_al_1990}).

\begin{newenv}
Rewriting as probability, we get that 
\begin{align*}
\hat{\pi}_1(n_1,\ldots,n_{\ell-1})
&:=\lim_{n_\ell\to \infty}\cdots \lim_{n_m\to \infty}\pi_1(n_1,\ldots,n_{m})\\
&=\P\Big(G_\Sigma<n_1\Big(1+\frac{G_2}{G_1}+\ldots+\frac{G_{\ell-1}}{G_1}\Big), \frac{G_{2}}{n_{2}}>\frac{G_{1}}{n_{1}},\ldots, \frac{G_{\ell-1}}{n_{\ell-1}}>\frac{G_{1}}{n_{1}}\Big).
\end{align*}
Next, we want to let \(n_1\to \infty\). We write \(G_\Sigma=G'_\Sigma+G'_1\), where \(G'_\Sigma\sim \operatorname{Gamma}(\sum_{k=2}^{\ell-1} n_k,1)\) and
 \(G'_1\sim \operatorname{Gamma}(n_1,1)\), and \(G_1,\ldots, G_{\ell-1},G'_1,G'_\Sigma\) are independent.
With this
\begin{align*}
&\Big\{G_\Sigma<n_1\Big(1+\frac{G_2}{G_1}+\ldots+\frac{G_{\ell-1}}{G_1}\Big)\Big\}\\
&=\Big\{G'_\Sigma+G'_1-n_1<\frac{n_1}{G_1}G_2+\ldots+\frac{n_1}{G_1}G_{\ell-1}\Big\}\\
&=\Big\{\frac{G'_\Sigma}{\sqrt n_1}+\sqrt n_1\Big(\frac{G'_1}{n_1}-1\Big)<\frac{\sqrt n_1}{G_1}G_2+\ldots+\frac{\sqrt n_1}{G_1}G_{\ell-1}\Big\}\,.
\end{align*}
Now we have, by Lemma \ref{thm:gamma-conv} and Proposition \ref{eq:gamma-clt},  
\[
\frac{G'_\Sigma}{\sqrt n_1}\to 0 \text{ and }\frac{\sqrt n_1}{G_1}G_k\to 0 \;  \text{in probability, and} \;
\sqrt n_1\Big(\frac{G'_1}{n_1}-1\Big)\stackrel{d}\to Z'
\]
for \(n_1\to \infty\) (where \(Z'\) is a standard normal random variable,
independent of \(G_2,\ldots,G_m\)). Using Slutsky's theorem, the tuple 
$$\left(\frac{G'_\Sigma}{\sqrt n_1}+\sqrt n_1\Big(\frac{G'_1}{n_1}-1\Big)-\Big(\frac{\sqrt n_1}{G_1}G_2+\ldots+\frac{\sqrt n_1}{G_1}G_{\ell-1}\Big),\frac{G_2}{n_2}-\frac{G_1}{n_1},\dotsc,\frac{G_{\ell-1}}{n_{\ell-1}}-\frac{G_1}{n_1}\right)$$
converges in distribution to $\Big(Z',\frac{G_2}{n_2}-1,\dotsc,\frac{G_{\ell-1}}{n_{\ell-1}}-1\Big)$, which is a continuous $\ell$-variate distribution. Hence,
\begin{align*}
\lim_{n_1\to\infty}\hat{\pi}_1(n_1,\ldots,n_{\ell-1})
&=\P(Z<0,G_2>n_2,\ldots, G_{\ell-1}>n_{\ell-1})\\
&=\P(Z<0)\prod_{k=2}^{\ell-1}\P(G_{k}>n_{k})=\frac 1 2\prod_{k=2}^{\ell-1}\P(G_{k}>n_{k}).
\end{align*}
\end{newenv}
Now, if we let the remaining arguments tend to \(\infty\) we get 
\[
\lim_{n_2\to\infty}\cdots\lim_{n_{\ell-1}\to\infty}\lim_{n_{1}\to\infty}\lim_{n_{\ell}\to\infty}\cdots\lim_{n_{m}\to\infty}\pi_1(n_1,\ldots,n_{m})
=\frac{1}{2^{\ell-1}}
\]
which coincides with the limit in  
\eqref{eq:gstoert} for \(\ell=m\). Note however, that in  
\eqref{eq:gstoert} we take the limit \(n_1\to\infty\) before all others, so this is a different limit!
\medskip

Finally, we consider the case where all \(n_1,\ldots,n_m\) tend to \(\infty\) simultaneously, while keeping their proportions constant.
That is, we study the winning probability for player 1 for  \((n_1,\ldots,n_m)=n(\alpha_1,\ldots,\alpha_m)\) when \(n\) tends to infinity.

\begin{proposition}\label{thm:gama-normal-win} Let \(\alpha_1,\ldots,\alpha_m>0\).
For every \(n\in\N_{\ge 1}\), let \(G_{1,n},\ldots,G_{m,n}\) be independent and let \(G_{k,n}\) have a \(\operatorname{Gamma}(\alpha_k n,1)\) distribution, \(k=1,\ldots,m\). Then 
\begin{align*}
\lim_{n\to \infty}\P\Big(\frac{G_{2,n}}{\alpha_2 n}>\frac{G_{1,n}}{\alpha_1 n}, \ldots , \frac{G_{m,n}}{\alpha_m n}>\frac{G_{1,n}}{\alpha_1 n}\Big)
&=\P\Big(\frac{Z_2}{\sqrt{\alpha_2 }}>\frac{Z_1}{\sqrt{\alpha_1 }}, \ldots , \frac{Z_m}{\sqrt{\alpha_m }}>\frac{Z_1}{\sqrt{\alpha_1 }}\Big),
\end{align*}
where \(Z_1,\ldots,Z_n\) are independent standard normal variables.
\end{proposition}

\begin{proof} 
Note that \(\frac{G_{k,n}}{\alpha_k n}>\frac{G_{1,n}}{\alpha_1 n}\Longleftrightarrow \sqrt{\alpha_k n}\big(\frac{G_{k,n}}{\alpha_k n}-1\big)>\sqrt{\frac{\alpha_k}{\alpha_1}}\sqrt{\alpha_1 n}\big(\frac{G_{1,n}}{\alpha_1 n}-1\big)\) so that with  
\(H_{k,n}:=\sqrt{\alpha_k n}\big(\frac{G_{k,n}}{\alpha_k n}-1\big)\),  \(k=1,\ldots,m\), 
\begin{align*}
&\P\Big(\frac{G_{2,n}}{\alpha_2 n}>\frac{G_{1,n}}{\alpha_1 n}, \ldots , \frac{G_{m,n}}{\alpha_m n}>\frac{G_{1,n}}{\alpha_1 n}\Big)
=\P\Big(H_{2,n}>\sqrt{\frac{\alpha_2}{\alpha_1}}H_{1,n}, \ldots , H_{m,n}>\sqrt{\frac{\alpha_m}{\alpha_1}}H_{1,n}\Big).
\end{align*}
By independence, the joint CDF of \(H_{1,n},\ldots,H_{m,n}\) is the product of the marginal CDFs, and 
each \(H_{k,n}\) converges to a standard normal variable in distribution, for \(n\to \infty\), in other words 
\begin{align*}
(H_{1,n},\dotsc,H_{m,n})\stackrel{d}{\to}(Z_1,\dotsc,Z_m)\quad\text{ as }n\to\infty.
\end{align*} Using the vector's distribution $P_{(H_{1,n},\dotsc,H_{m,n})}$, we can write
\begin{align*}
&\P\Big(H_{2,n}>\sqrt{\frac{\alpha_2}{\alpha_1}}H_{1,n}, \ldots , H_{m,n}>\sqrt{\frac{\alpha_m}{\alpha_1}}H_{1,n}\Big)
=\\
&\int_{\R^m} 1_{\big\{x_2>\sqrt{\frac{\alpha_k}{\alpha_1}} x_1,\dotsc,x_m>\sqrt{\frac{\alpha_m}{\alpha_1}} x_1\big\}} d P_{(H_{1,n},\dotsc,H_{m,n})}(x_1,\dotsc,x_m).
\end{align*}
Knowing that for the $m$-variate standard normal distribution $P_{(Z_1,\dotsc,Z_m)}$ all Borel sets are continuity sets \cite[Theorem 29.2]{billingsley1986probability}, we get the convergence
\begin{align*}
&\lim_{n\to\infty}\int_{\R^m} 1_{\big\{x_2>\sqrt{\frac{\alpha_2}{\alpha_1}} x_1,\dotsc,x_m>\sqrt{\frac{\alpha_m}{\alpha_1}} x_1\big\}} d P_{(H_{1,n},\dotsc,H_{m,n})}(x_1,\dotsc,x_m)\\
&=\int_{\R^m} 1_{\big\{x_2>\sqrt{\frac{\alpha_2}{\alpha_1}} x_1,\dotsc,x_m>\sqrt{\frac{\alpha_m}{\alpha_1}} x_1\big\}} d P_{(Z_1,\dotsc,Z_m)}(x_1,\dotsc,x_m),
\end{align*}
which can also be written as
\begin{align*}
\int_{\R^m} \prod_{k=2}^m 1_{\big\{x_k>\sqrt{\frac{\alpha_k}{\alpha_1}} x_1\big\}} d P_{Z_2}(x_2)\cdots d P_{Z_m}(x_m),
\end{align*}
or as $\P\Big(\frac{Z_2}{\sqrt{\alpha_2 }}>\frac{Z_1}{\sqrt{\alpha_1 }}, \ldots , \frac{Z_m}{\sqrt{\alpha_m }}>\frac{Z_1}{\sqrt{\alpha_1 }}\Big),$ proving the assertion.
\end{proof}

Now we use Proposition \ref{thm:gama-normal-win}, disintegration over \(Z_1\), and independence 
to compute the limiting winning probability for player 1: 
\begin{align*}
&\lim_{n\to \infty}\P\Big(\frac{G_2}{\alpha_k n}>\frac{G_1}{\alpha_1 n}, \ldots , \frac{G_m}{\alpha_k n}>\frac{G_1}{\alpha_1 n}\Big)\\
&=\P\Big(\frac{Z_2}{\sqrt{\alpha_2 }}>\frac{Z_1}{\sqrt{\alpha_1 }}, \ldots , \frac{Z_m}{\sqrt{\alpha_m }}>\frac{Z_1}{\sqrt{\alpha_1 }}\Big)
=\int_{-\infty}^\infty \P\Big(\frac{Z_2}{\sqrt{\alpha_2 }}>\frac{z}{\sqrt{\alpha_1 }}, \ldots , \frac{Z_m}{\sqrt{\alpha_m }}>\frac{z}{\sqrt{\alpha_1 }}\Big)\frac{1}{\sqrt{2\pi}}e^{-\frac{z^2}{2}}dz\\
&=\int_{-\infty}^\infty \prod_{k=2}^m\P\Big(\frac{Z_k}{\sqrt{\alpha_k }}>\frac{z}{\sqrt{\alpha_1 }}\Big)\frac{1}{\sqrt{2\pi}}e^{-\frac{z^2}{2}}dz
=\int_{-\infty}^\infty \prod_{k=2}^m\Big(1-\Phi\Big(\sqrt{\frac{\alpha_k}{\alpha_1}}\ z\Big)\Big)\frac{1}{\sqrt{2\pi}}e^{-\frac{z^2}{2}}dz\,.
\end{align*}

\begin{remark}
Recall our geometric interpretation of the winning probability from Remark \ref{rem:cones}. In this interpretation, 
Theorem \ref{thm:gama-normal-win} describes the limiting case when the cone \(C_W\) is exited through the face
corresponding to player 1. For large \(n\), the image of the random walk converges to the line \(\{t (n_1,\ldots,n_m):t\in [0,\infty)\}\) (which is the main diagonal of the cuboid \(C_W\cap [0,\infty)^m\)) almost surely. 
The ratios between the axes of the cuboid are given by those of the \(\alpha_i\)s, \(\frac{\alpha_\ell}{\alpha_k}=\frac{n_\ell}{n_k}\) for all \(k,\ell\).
Through which of the faces the walk exits is determined by `residual fluctuations' around the diagonal which are approximately normal.
\end{remark}

Note that by virtue of Theorem \ref{thm:rep-pi1} we may extend the function \(\pi_1\) to real positive arguments. With this we can summarize
the preceding argument as follows:
\begin{theorem}
Let \(\alpha_1,\ldots,\alpha_m>0\). Then 
\[
\lim_{n\to \infty}\pi_1(\alpha_1 n,\ldots, \alpha_m n)=\int_{-\infty}^\infty \prod_{k=2}^m\Big(1-\Phi\Big(\sqrt{\frac{\alpha_k}{\alpha_1}}\ z\Big)\Big)\frac{1}{\sqrt{2\pi}}e^{-\frac{z^2}{2}}dz\,.
\]
\end{theorem}

\new{For $m=2, 3$, the integral can be explicitly calculated. In the special case $m=2$, we obtain $\lim_{n\to \infty}\pi_1(\alpha_1 n,\alpha_2 n)=\frac{1}{2}$ for all $\alpha_1,\alpha_2>0$.}

\section{Being last in the game}\label{sec:losing-prob}

Now we consider the complementary problem of computing the probability of a given player to reach their goal
after all other players.

We can squeeze out quite a bit more from Lemma \ref{lem:conjugate}, namely the probability of finishing before players \(2,\ldots,m-1\), but after 
player \(m\).
\begin{align*}
&\mathbb{P}(A_2\leq n_2-1,\dotsc,A_{m-1}\leq n_{m-1}-1,A_m\ge n_m)\\
&=\mathbb{P}(A_2\leq n_2-1,\dotsc,A_{m-1}\leq n_{m-1}-1)\\
&\quad -\mathbb{P}(A_2\leq n_2-1,\dotsc,A_{m-1}\leq n_{m-1}-1,A_m\le n_m-1)\\
&\stackrel{(\ast)}{=}\mathbb{P}\left(X_2>\frac{p_2}{p_1},\dotsc,X_{m-1}>\frac{p_{m-1}}{p_1}\right)-\mathbb{P}\left(X_2>\frac{p_2}{p_1},\dotsc,X_m>\frac{p_m}{p_1}\right)\\
&=\mathbb{P}\left(X_2>\frac{p_2}{p_1},\dotsc,X_{m-1}>\frac{p_{m-1}}{p_1},X_m\le \frac{p_m}{p_1}\right)
\end{align*}
However, for this innocent looking calculation we need to know that \((\ast)\) is true, i.e., 
\[
\mathbb{P}(A_2\leq n_2-1,\dotsc,A_{m-1}\leq n_{m-1}-1)=\mathbb{P}\left(X_2>\frac{p_2}{p_1},\dotsc,X_{m-1}>\frac{p_{m-1}}{p_1}\right).
\] 
That this equality holds can be seen by recalling that \((X_2,\ldots,X_{m-1})\sim\operatorname{ID}(n_2,\ldots,n_{m-1};n_1)\) (cf.~Remark \ref{rem:idir}\eqref{it:idir-marg})
and noting that \((A_2,\ldots, A_{m-1})\) has negative multinomial distribution with parameters \(n_1,\frac{p_2}{1-p_m},\ldots,\frac{p_{m-1}}{1-p_m}\), so that 
\begin{align*}
&\mathbb{P}(A_2\leq n_2-1,\dotsc,A_{m-1}\leq n_{m-1}-1)\\
&=\mathbb{P}\Big(X_2>\frac{\frac{p_2}{1-p_m}}{\frac{p_1}{1-p_m}},\ldots,X_{m-1}>\frac{\frac{p_{m-1}}{1-p_m}}{\frac{p_1}{1-p_m}}\Big)
=\mathbb{P}\left(X_2>\frac{p_2}{p_1},\dotsc,X_{m-1}>\frac{p_{m-1}}{p_1}\right).
\end{align*}
\detail{ 
\begin{align*}
&\mathbb{P}(A_2\leq n_2-1,\dotsc,A_{m-1}\leq n_{m-1}-1)\\
&=p_1^{n_1}\sum_{k_2=0}^{n_2-1}\cdots\sum_{k_{m-1}=0}^{n_{m-1}-1}\sum_{k_m=0}^{\infty}\binom{n_1+k_2+\dotsb+k_m-1}{n_1-1, k_2,\dotsc ,k_m}\prod_{\ell=2}^m p_\ell^{k_\ell}\\
&=p_1^{n_1}\sum_{k_2=0}^{n_2-1}\cdots\sum_{k_{m-1}=0}^{n_{m-1}-1}\binom{n_1+k_2+\dotsb+k_{m-1}-1}{n_1-1, k_2,\dotsc ,k_{m-1}}\prod_{\ell=2}^{m-1} p_\ell^{k_\ell}\sum_{k_m=0}^{\infty}\binom{n_1+k_2+\dotsb+k_m-1}{k_m}p_m^{k_m}\\
&=p_1^{n_1}\sum_{k_2=0}^{n_2-1}\cdots\sum_{k_{m-1}=0}^{n_{m-1}-1}\binom{n_1+k_2+\dotsb+k_{m-1}-1}{n_1-1, k_2,\dotsc ,k_{m-1}}\bigg(\prod_{\ell=2}^{m-1} p_\ell^{k_\ell}\bigg)\Big(\frac{1}{1-p_m}\Big)^{n_1+k_2+\dotsb+k_{m-1}}\\
&=\Big(\frac{p_1}{1-p_m}\Big)^{n_1}\sum_{k_2=0}^{n_2-1}\cdots\sum_{k_{m-1}=0}^{n_{m-1}-1}\binom{n_1+k_2+\dotsb+k_{m-1}-1}{n_1-1, k_2,\dotsc ,k_{m-1}}\prod_{\ell=2}^{m-1} \Big(\frac{p_\ell}{1-p_m}\Big)^{k_\ell}\\
&=\mathbb{P}\Big(X_2>\frac{\frac{p_2}{1-p_m}}{\frac{p_1}{1-p_m}},\ldots,X_{m-1}>\frac{\frac{p_{m-1}}{1-p_m}}{\frac{p_1}{1-p_m}}\Big)
=\mathbb{P}\left(X_2>\frac{p_2}{p_1},\dotsc,X_{m-1}>\frac{p_{m-1}}{p_1}\right).
\end{align*}}

Repeating this calculation for \(m-1,\ldots,2\) we get 
\begin{equation}\label{eq:lemma13var}
\mathbb{P}(A_2\leq n_2-1,\dotsc,A_{m}\leq n_{m}-1)=\mathbb{P}\left(X_2>\frac{p_2}{p_1},\dotsc,X_{m}>\frac{p_{m}}{p_1}\right).
\end{equation}

By \eqref{eq:lemma13var} we obtain the following result, which is proven in the same way as Theorem \ref{thm:pi-bij},  with the integration symbols $\int_{\frac{p_\ell}{p_1}}^\infty$ 
changed to $\int_{0}^\frac{p_\ell}{p_1}$.
\begin{theorem}
Let \(\Delta_{m-1}\) denote the \(m-1\) dimensional simplex in \(\R^m\).
Then the function $\tau\colon \Delta_{m-1}\to\Delta_{m-1}, p=(p_1,\dotsc,p_m)\mapsto(\tau_1(p),\dotsc,\tau_m(p))$ is a bijection.
The restriction of  $\pi$ to the relative interior of \(\Delta_{m-1}\) is a diffeomorphism. 
\end{theorem}

Another consequence of \eqref{eq:lemma13var} is the following counterpart of Theorem \ref{thm:rep-pi1}.

\begin{proposition}\label{thm:rep-tau1}
Let \(m\ge 2\).  For all \(n_1,\ldots,n_m\in \N_{\ge 1}\) consider 
\[
\tau_1(n_1,\ldots,n_m)=\mathbb{P}(A_2\geq n_2,\dotsc,A_m\geq n_m) 
\]
where $(A_2,\dotsc,A_m)$ have a negative multinomial distribution with parameters $(n_1;p_2,\dotsc,p_m)$ 
and \(p_\ell=\frac{n_\ell}{n_1+\ldots,n_m}\) for \(\ell\in\{2,\ldots,m\}\).

Then we have the following representations:
\begin{enumerate}
\item \label{it:rep-tau1-1}
\(\tau_1(n_1,\ldots,n_m)=\mathbb{P}\left(X_2\le\frac{n_2}{n_1},\dotsc,X_m\le\frac{n_m}{n_1}\right)\)\\
where $(X_2,\dotsc,X_m)\sim \operatorname{ID}(n_2,\dotsc,n_m;n_1)$.
\item \label{it:rep-tau1-2} \(\tau_1(n_1,\ldots,n_m)=\mathbb{P}\left(\frac{G_2}{n_2} \le \frac{G_1}{n_1},\dotsc,\frac{G_m}{n_m} \le \frac{G_1}{n_1}\right)\)\\
where \(G_1,\ldots,G_m\) are independent random variables with \(G_j\sim \operatorname{Gamma}(n_j,1)\), \(j=1,\ldots,m\).
\end{enumerate}
\end{proposition}

\begin{remark}
Of course the inequality signs on the respective right hand sides of Theorem \ref{thm:rep-tau1} \eqref{it:rep-tau1-1} and \eqref{it:rep-tau1-2}
can be replaced by strict inequality signs, since 
the inverted Dirichlet distribution is a continuous distribution with a probability density.
\end{remark}

\begin{remark}
With the same arguments as the ones leading to the conclusion of Theorem \ref{thm:rep-tau1} one can calculate the probability that player 1 finishes after players \(\ell_1,\ldots,\ell_{M-1}\), but before
players  \(\ell_{M},\ldots,\ell_m\) and {\em a fortiori}, the probability of winning the \(M\)-th place by summing over all
subsets of \(\{2,\ldots,m\}\) of cardinality \(M-1\).
\end{remark}

Invoking Theorem \ref{thm:rep-tau1}\eqref{it:rep-tau1-2} and repeating the calculations leading to equation \eqref{eq:Pdisint}, we get 
a corresponding equation for the probability of player 1 finishing last:
\begin{align}
\label{eq:Pdisint2}
\mathbb{P}(A_2\geq n_2,\dotsc,A_m\geq n_m)
&=\int_{0}^\infty \prod_{k=2}^m\mathbb{P}\left(\frac{G_k}{n_k}\le x\right)n_1^{n_1}\frac{x^{n_1-1}e^{-n_1 x}}{\Gamma(n_1)}dx,
\end{align}
where \(G_k\sim \operatorname{Gamma}(n_k,1)\) for all \(k\in \{2,\ldots,m\}\).
In fact, \eqref{eq:Pdisint2} was used to compute the probabilities in Table \ref{tbl:dice-prob-last}. 

\medskip

\noindent\textbf{Tackling Question \ref{question2}.} Imitating the calculations from Section \ref{sec:winning} one can proceed further, however, the analog of  Lemma \ref{th:inc-reg-beta-dec}
which one would need to show that the losing probability is increasing in \(n_m\) is false in general. In fact, it turns out that 
in general the answer to Question \ref{question2} is `No'.

\begin{example}  
We provide a counterexample for monotonicity of \(\tau_1\) with respect to \(n_3\): 
\[
\tau_1(1,1,1)= 0.333333> 0.319444 = \tau_1(1,1,2).
\]
So, in general, \(\tau_1\) is not increasing in the second to \(m\)-th argument. 

For \((n_1,n_2,n_3)=(1,2,3)\) or \((n_1,n_2,n_3)=(3,4,5)\) we get \(\tau_1<\tau_2<\tau_3\). 
For \((n_1,n_2,n_3)=(4,5,6)\) we get \(\tau_3<\tau_2<\tau_1\). 

For all examples with \(m=3\) we observed that if \(n_1<n_2<n_3\), then the probability \(\tau_2\) for player 2 being last, is never maximal among the 3 players. However, it can be minimal, for example, if \((n_1,n_2,n_3)=(3,4,6)\). 

For 4 players, choosing \((n_1,n_2,n_3,n_4)=(1, 1, 2, 4)\), player 3 has the highest probability of being last,
although they have neither the largest nor the smallest advancement probability.

In conclusion, the behavior of the losing probabilities is hard to summarize for the general game. 
The biblical proverb `So the last will be first, and the first will be last' (Matthew 20:16) is true (in the probabilistic sense) for the original 11 player-game described in the introduction, but is false for the general game.
\end{example}

We turn to asymptotics. Proposition \ref{thm:rep-tau1} can be used in the same way as in Section \ref{sec:asymptotics} to compute the limiting 
probability of player one finishing last when one or several of the \(n_k\) tend to \(\infty\). We get the following results:

\begin{theorem}\label{thm:tau-asympotics} Let \(\tau_1\) be as in Proposition \ref {thm:rep-tau1}. Then
\begin{enumerate}
\item \(
\lim_{n_1\to \infty}\tau_1(n_1,\ldots,n_m)=\prod_{k=2}^m\mathbb{P}(G_k \le  n_k)\,.
\)
\item With \(X_2,\ldots,X_m\) as in Proposition \ref {thm:rep-tau1}\eqref{it:rep-tau1-1},  
\begin{align*}
&\lim_{n_m\to \infty}\tau_1(n_1,n_2,\ldots,n_m)\\
&=\int_0^{\frac{n_2}{n_1}}\cdots\int_0^{\frac{n_{m-1}}{n_1}}\int_{n_1(1+\sum_{k=2}^{m-1}s_k)}^\infty f_\Gamma(s_m)ds_{m}\ \varphi_{<m}(s_2,\ldots,s_{m-1})ds_{m-1}\dotsc ds_2,
\end{align*}
where 
\(\varphi_{<m}\) is the marginal density of \(X_2,\ldots,X_{m-1}\) and \(f_\Gamma\) is a \(\operatorname{Gamma}\Big(\sum_{k=1}^{m-1}n_k,1\Big)\) density.
\item\label{item:tau-alphas} For given \(\alpha_1,\ldots,\alpha_m>0\)  
\[
\lim_{n\to \infty}\tau_1(\alpha_1 n,\ldots, \alpha_m n)=\int_{-\infty}^\infty \prod_{k=2}^m \Phi\Big(\sqrt{\frac{\alpha_k}{\alpha_1}}\ z\Big)\frac{1}{\sqrt{2\pi}}e^{-\frac{z^2}{2}}dz\,.
\]
\end{enumerate}
\end{theorem}

\medskip

In view of Theorem \ref{thm:tau-asympotics} \eqref{item:tau-alphas}
above, the asymptotic probability for player 1 to be last is $\E\Big[\prod_{k=2}^m\Phi\Big(\sqrt{\frac{\alpha_k}{\alpha_1}}Z\Big)\Big]$ for a standard normal variable $Z$.
Differentiating with respect to some \(\alpha_\ell\), where $\ell\in\{2,\dotsc,m\}$, yields
\begin{align*}
\frac{1}{2\sqrt{\alpha_1\alpha_\ell}}\E\Big[\prod_{\substack{k=2\\k\neq \ell}}^m\Phi\Big(\sqrt{\frac{\alpha_k}{\alpha_1}}Z\Big)\phi\Big(\sqrt{\frac{\alpha_\ell}{\alpha_1}}Z\Big)Z\Big],
\end{align*}
where $\phi$ denotes the standard normal density. By the symmetry of $\phi$, the above expectation is
\begin{align*}
\E\bigg[1_{\{Z>0\}}\bigg(\prod_{\substack{k=2\\k\neq \ell}}^m\Phi\Big(\sqrt{\frac{\alpha_k}{\alpha_1}}Z\Big)-\prod_{\substack{k=2\\k\neq \ell}}^m\Phi\Big(-\sqrt{\frac{\alpha_k}{\alpha_1}}Z\Big)\bigg)\phi\Big(\sqrt{\frac{\alpha_\ell}{\alpha_1}}Z\Big)Z\bigg],
\end{align*}
and since for $w>0$, we have $\Phi(w)>\frac{1}{2}>\Phi(-w)=1-\Phi(w)$, the above integrand is strictly greater than zero  except for the case $m=2$, where it vanishes.
Hence, $\lim_{n\to \infty}\tau_1(\alpha_1 n,\ldots, \alpha_m n)$ strictly increases in each $\alpha_\ell$, $2\leq \ell\leq m$, with the exception of $m=2$, where it is constant at $\frac{1}{2}$. In other words, the limiting behavior of the probabilities of being last is the same as of the winning probabilities. Therefore, taking an asymptotic viewpoint, we can answer Question 2 with `Yes' again.

\medskip

\noindent\textbf{Acknowledgements:} We thank Evita Lerchenberger and Christoph Oberbucher (both FDZ University of Graz) for bringing the problem to our attention.

\end{document}